\documentclass[11pt,reqno]{amsart}
\usepackage{fullpage}
\usepackage[mathscr]{eucal}
\usepackage{mathrsfs}
\usepackage{amsfonts}
\usepackage{amsmath}
\usepackage{color}
\usepackage{amsthm}
\usepackage{amssymb}
\usepackage{latexsym}
\usepackage[noautoscale,enableskew]{youngtab}
\usepackage[all]{xy}
\usepackage[mathscr]{eucal}
\usepackage{palatino}
\usepackage{fullpage}
\usepackage{verbatim}

\theoremstyle{plain}
\newtheorem{thm}[equation]{Theorem}
\newtheorem{cor}[equation]{Corollary}
\newtheorem{prop}[equation]{Proposition}
\newtheorem{lem}[equation]{Lemma}

\theoremstyle{definition}
\newtheorem{defn}[equation]{Definition}

\theoremstyle{remark}

\newtheorem{examp}[equation]{Example}
\newtheorem{rem}[equation]{Remark}
\newtheorem{assume}[equation]{Assumption}

\makeatletter
\renewcommand{\subsection}{\@startsection{subsection}{2}{0pt}{-3ex
plus -1ex minus -0.2ex}{-2mm plus -0pt minus
-2pt}{\normalfont\bfseries}} \makeatother

\numberwithin{equation}{subsection}

\newcommand{\Lmod}[1]{#1\text{-}{\mathsf{mod}}}

\newcommand{\hdot}{{\:\raisebox{2pt}{\text{\circle*{1.5}}}}}
\newcommand{\idot}{{\:\raisebox{2pt}{\text{\circle*{1.5}}}}}

\DeclareMathOperator{\res}{{\mathrm{res}}}

\DeclareMathOperator{\sym}{\mathrm{Sym}}

\DeclareMathOperator{\Ker}{\mathrm{Ker}}

\DeclareMathOperator{\End}{\mathrm{End}}

\DeclareMathOperator{\rk}{\mathrm{rk}}

\DeclareMathOperator{\gr}{\mathrm{gr}}
\DeclareMathOperator{\hilb}{{\mathrm{Hilb}}}

\DeclareMathOperator{\rad}{\mathrm{rad}}

\DeclareMathOperator{\Lie}{\mathrm{Lie}}
\DeclareMathOperator{\Tr}{\mathrm{Tr}}

\DeclareMathOperator{\ggr}{{\widetilde{\mathrm{gr}}}}

\DeclareMathOperator{\Ad}{\mathrm{Ad}}
\DeclareMathOperator{\ad}{\mathrm{ad}}

\newcommand{\beq}{\begin{equation}\label}
\newcommand{\eeq}{\end{equation}}

\renewcommand{\o}{\otimes }
\newcommand{\bplus}{\mbox{$\bigoplus$}}
\newcommand{\cc}{{\scr C}}
\newcommand{\kap}{{\kappa}}

\newcommand{\cchi}{{\boldsymbol{\chi}}}

\renewcommand{\t}{{\mathfrak t}}

\newcommand{\wt}{\widetilde }

\newcommand{\zz}{{\mathfrak C}}
\newcommand{\xx}{{\mathfrak X}}

\newcommand{\erem}{\hfill$\lozenge$\end{rem}\vskip 3pt }

\newcommand{\scr}[1]{\mathscr{#1}}

\newcommand{\be}{\beta }

\newcommand{\gl}{{\mathfrak g\mathfrak l}}

\newcommand{\hodge}{{{\operatorname{Hodge}}}}

\newcommand{\dd}{{\mathscr{D}}}
\newcommand{\oo}{{\mathcal{O}}}

\renewcommand{\aa}{{\mathscr{A}}}

\renewcommand{\tt}{{\mathfrak T}}

\newcommand{\dcoh}{D^b_{\operatorname{coh}}}
\newcommand{\pa}{\partial }

\renewcommand{\gg}{{\mathfrak G}}

\newcommand{\hc}{{Harish-Chandra }}

\newcommand{\C}{\mathbb{C}}
\newcommand{\g}{\mathfrak{g}}
\renewcommand{\b}{\mathfrak{b}}

\newcommand{\norm}{{\operatorname{norm}}}

\newcommand{\pp}{{\scr P}}
\newcommand{\rr}{{\scr R}}

\newcommand{\bb}{{\scr B}}

\newcommand{\mm}{{\mathcal M}}

\newcommand{\T}{{\mathcal T}}
\newcommand{\inv}{^{-1}}

\newcommand{\Z}{{\mathbb Z}}
\newcommand{\en}{{\enspace}}

\newcommand{\sset}{\subset}

\newcommand{\mto}{\mapsto}

\newcommand{\N}{{\mathcal{N}}}
\newcommand{\tgg}{{\wt{\mathfrak G}}}
\newcommand{\mmu}{{\boldsymbol{\mu}}}

\renewcommand{\sl}{{\mathfrak s\mathfrak l}}
\newcommand{\biP}{{\mathcal P}}
\DeclareMathOperator{\Irr}{\mathrm{Irr}}
\newcommand{\eul}{\mathsf{e}\mathsf{u}}
\newcommand{\Schur}{\mathbf{S}}

\newcommand{\adh}{{\mathbf{h}}}

\newcommand{\ds}{{\dots}}
\newcommand{\pnps}{{principal nilpotent pairs }}
\newcommand{\pnp}{{principal nilpotent pair }}
\newcommand{\nw}{{\mathbf{nw}}} 
\newcommand{\se}{{\mathbf{se}}}
\newcommand{\vr}{{\varrho}}

\renewcommand{\cc}{{\C^\times\times\C^\times}}
\newcommand{\bu}{{\bullet}}
\newcommand{\s}{\mathfrak{S}}
\newcommand{\core}{\mathrm{core}}
\newcommand{\quo}{\mathrm{quo}}
\newcommand{\affs}{\widehat{\mathfrak{S}}}
\newcommand{\Pnone}{\mathcal{P}_{n-1}}
\newcommand{\Pn}{\mathcal{P}}

\newcommand{\id}{\mathrm{id}}
\newcommand{\Lsimp}{{\mathscr{L}}}
\newcommand{\IC}{IC}
\newcommand{\Oflat}{\widetilde{\theta}_{\xx}}
\newcommand{\xxnil}{\widetilde{\xx}^{\mathrm{nil}}}

\newcommand{\Hilb}{\hilb^n \C^2}
\newcommand{\taut}{\mathcal{V}}
\newcommand{\Nat}{\mathbb{N}}
\newcommand{\acoe}{\mathbf p}

\newcommand{\uf}{\mathfrak{u}}
\newcommand{\Q}{\mathbb{Q}}
\newcommand{\zf}{\mathfrak{z}}
\renewcommand{\be}{\mathbf{e}}
\newcommand{\Rbox}{\mathrm{R}}
\newcommand{\Cbox}{\mathrm{C}}
\newcommand{\Pprod}{\mathrm{c}}
\newcommand{\emptybox}{}
\newcommand{\Cs}{\C^\times}

\newcommand{\SYT}{\mathrm{SYT}}
\newcommand{\YT}{\mathrm{YT}}
\newcommand{\hcf}{\mathrm{h.c.f.} \,}
\newcommand{\Jset}{\mathcal{J}_n}
\newcommand{\Rcore}[2]{\mho^{#1}_{#2}}
\newcommand{\Gpoly}[2]{\mathcal{G}^{#1}_{#2}}
\newcommand{\Fqt}{\mathbb{F}}
\newcommand{\halfZ}{\frac{1}{2} \Z}

\newcommand{\bbe}{{\mathbf e}}

\newcommand{\zzr}{\zz^r}
\newcommand{\ggreg}{\gg^r}
\newcommand{\tggr}{\tgg^r}
\newcommand{\tggrA}{\tgg^{r,A}}
\newcommand{\zzrA}{\zz^{r,A}}
\newcommand{\ggregA}{\gg^{r,A}}
\newcommand{\mbf}{\mathbf}
\newcommand{\Snw}{\mathbf{S}_{\nw}}
\newcommand{\Sse}{\mathbf{S}_{\se}}
\newcommand{\one}{\mathbf{1}}

\newcommand{\rrnil}{{\rr}^{\mathrm{nil}}}
\newcommand{\zznil}{\widetilde{\zz}^{\mathrm{nil}}}
\newcommand{\mmnil}{{\mm^{\mathrm{nil}}}}
\newcommand{\gs}{\mathbf{d}}
\newcommand{\gss}{\mathbf{c}}
\newcommand{\Ozflat}{\widetilde{\theta}_{\zz}}

\newcommand{\Ch}[2]{\mathsf{H}_{#1,#2}}
\newcommand{\catO}[2]{\mathcal{O}_{#1,#2}}
\def\hp{\hphantom{x}}
\newcommand{\can}{\mathrm{can}}

\definecolor{gray1}{gray}{0.4}

\begin{document}
\title{{\large{\textsf{Some combinatorial identities related \\
\vskip 2pt
to commuting varieties and Hilbert schemes}}}}

\thanks{The first author would like to thank Bernard Leclerc and Iain Gordon for stimulating discussions. This work is the result of a visit of the first author to the University of Chicago, made possible through a Cecil-King travel scholarship. The first author would like to thank the London Mathematical Society and Cecil-King Foundation for this opportunity and the University of Chicago for its hospitality and support. The research of the first author was supported through the programme ``Oberwolfach Leibinz Fellows'' by the Mathematisches Forshungsinstitut Oberwolfach in 2010. The research of the second author was supported in part by the NSF award DMS-1001677.}

\author{Gwyn Bellamy and Victor Ginzburg\\\vskip0pt
{with an appendix by\\ Eliana Zoque}}
\address{\textbf{G.B.} : Max-Planck-Institut f\"ur Mathematik, Vivatsgasse 7, 53111 Bonn, Germany.}
\email{gwyn@mpim-bonn.mpg.de}
\address{\textbf{V.G.} : Department of Mathematics, University of Chicago,  Chicago, IL 
60637, USA.}
\email{ginzburg@math.uchicago.edu}
\address{\textbf{E.Z.} : Department of Mathematics, UC Riverside, USA.}
\email{elizoque@ucr.edu}

\begin{abstract} 
In this article we explore some of the combinatorial consequences of
recent results  relating the isospectral commuting variety and the Hilbert scheme of
 points in the plane.
\end{abstract}

\maketitle

\centerline{\sf Table of Contents}

$\hspace{40mm}$ {\footnotesize \parbox[t]{115mm}{
\hp${}_{}$\!\hp1.{ $\;\,$} {Introduction}\newline
\hp2.{ $\;\,$} {Bigraded $G$-character of $\rr$}\newline
\hp3.{ $\;\,$} {Principal nilpotent pairs}\newline
\hp4.{ $\;\,$} {Principal nilpotent pairs for $\mathfrak{gl}_n$}\newline
\hp5.{ $\;\,$} {Polygraph spaces}\newline
\hp6.{ $\;\,$} {Rational Cherednik algebras}\newline
\hp7.{ $\;\,$} {The Harish-Chandra module and Cherednik algebras}\newline
\hp8.{ $\;\,$} {Appendix by E. Zoque:} \newline
\hp${}_{}$ { $\;\;\,$} {$T$-orbits of \pnps}
}

\bigskip

\normalsize

\section{Introduction}

In this paper, we derive various combinatorial identities
by comparing bigraded characters of objects of four
different types. Fix an integer $n\geq 1$ and let
$\g=\gl_n$.

The objects of the first type
are associated with the {\em Procesi bundle} $\pp$
on the Hilbert scheme of $n$ points in the plane. The  Procesi bundle
was introduced and studied by M. Haiman in his work on
the $n!$ theorem \cite{HaimanJAMS}-\cite{HaimanSurvey}.
According to Haiman, the combinatorics of
the  Procesi bundle is closely related to Macdonald
polynomials. 

The objects of the second type
are associated with a certain remarkable
coherent sheaf $\rr$ on (the normalization of) the
commuting variety of the Lie algebra $\gl_n$, introduced by one of us in \cite{Iso}.
The sheaf $\rr$ has an interpretation
in terms of a certain double analogue of the Grothendieck-Springer
resolution, to be recalled in \S\ref{tx} below. Therefore, the
combinatorics of the coherent sheaf $\rr$ is related
to the geometry of the flag variety of $\gl_n$ and
to the standard
combinatorics of root systems.
Now, it was explained in  \cite{Iso} how 
one can use the sheaf $\rr$ to construct (a close cousin of) the 
Procesi bundle $\pp$. 
This yields,
on the combinatorial side, various identities relating
the combinatorics of the root system of  $\gl_n$  to 
Kostka-Macdonald
polynomials. 

The objects of the third type are associated with $D$-modules on 
the Lie algebra $\gl_n$. There is a distinguished
$D$-module $\mm$, called
the Harish-Chandra module, that has played a key
role in \cite{Iso}. The Harish-Chandra module
$\mm$ comes equipped with a canonical Hodge filtration
and the sheaf $\rr$ is obtained, essentially, as
an associated graded sheaf $\gr^\hodge\mm$.
Thus, the sheaves  $\rr$ and $\mm$ have closely
related character formulas. 

Finally, the objects of the fourth type  are associated with
representations of {\em rational Cherednik algebras}. Specifically, in section \ref{sec:rationalCherednik} we are interested in  character formulas for simple objects in the category $\oo$ for
rational Cherednik algebras of type $\mathbf A$. 
These  character formulas can be obtained,
thanks to the work of Rouquier \cite{RouquierQSchur},
from the multiplicity numbers of simple modules in standard modules for Schur algebras. 
The  latter  may be expressed in
terms of  Kazhdan-Lusztig type polynomials 
associated with canonical bases in a Fock space,
by the work of Leclerc-Thibon \cite{LT}-\cite{LTLR}
and Varagnolo-Vasserot \cite{VaragnoloVasserotDecomposition}.

On the other hand, many simple objects
of the category $\oo$ for the
rational Cherednik algebra can be constructed
by applying a version of the Hamiltonian reduction  functor
introduced by Calaque, Enriquez, and Etingof \cite{CEE}
to various direct summands of the Harish-Chandra module.
This relates the characters of the simple objects
to the characters of the Harish-Chandra module.
Thus, combining everything together, we obtain
in section \ref{sec:Dmod} an interesting identity that involves some Kazhdan-Lusztig  polynomials on one side and some Macdonald polynomials on the other side.\\

We outline in a bit more detail the main results of each section.\\

In section \ref{sec:bigraded} we introduce the sheaf $\rr$ on the normalization of the commuting variety and describe its $G \times W \times \cc$-equivariant structure. In the main result of this section, Theorem \ref{character_2}, we give a formula for the bigraded $G$-character of the global sections of $\rr$ in terms of certain degenerate Macdonald polynomials and a bivariate analogue of Kostant's partition function. This formula can be interpreted as a double analogue of Hesselink's formula for the graded $G$-character of $\C[\g]$. 

The fixed (up to conjugation by $G$) points in the commuting variety with respect to action of $\cc$ are the ``principle nilpotent pairs'' of that variety. The fiber of $\rr$ at each of these fixed points is a bigraded space of dimension $|W|$. The goal of section \ref{sec:pnp} is to give a formula, Theorem \ref{thm:pnpcharacter}, for the bigraded character of each of these special fibers. The proof of this theorem is an intricate calculation in equivariant $K$-theory, which uses in an essential way the alternative description given in \cite{Iso} of $\rr$ as a complex of sheaves on a doubled analogue of the Grothendieck-Springer resolution. The results of sections \ref{sec:bigraded} and \ref{sec:pnp} are valid for any connected complex reductive group $G$.   

In section \ref{sec:glnpnp} we show that it is possible to give an explicit combinatorial expression for the formula of Theorem \ref{thm:pnpcharacter} when $G = GL_n$. In this case it is known from \cite{Iso} and \cite{GordonMacdonald} that the fibers of the sheaf $\rr$ at the principle nilpotent pairs are isomorphic to the fibers of a $\cc$-equivariant sheaf, closely related to the Procesi bundle, at the fixed points of the Hilbert scheme. Therefore the bigraded character of these fibers is also given by a recursive formula of Garsia and Haiman, based on the Pieri rules for transformed Macdonald polynomials. We show by direct computation that our combinatorial expression is equivalent to Garsia and Haiman's formula.

This remarkable relationship between the sheaf $\rr$ and the Procesi bundle on the Hilbert scheme is exploited in section \ref{sec:polygraph} in order to describe more completely the full $G \times W \times \cc$-equivariant structure of $\rr$. For almost all irreducible $G$-representations $V_\mu$, we show in Theorem \ref{thm:Pone} that the bigraded $W$-character of the $V_\mu$-isotypic component of $\rr$ is expressible in terms of transformed Macdonald polynomials. Corollary \ref{cor:zznormchar} gives a similar formula for the $G$-isotypic components of the normalized commuting variety.\\

In section \ref{sec:rationalCherednik}, we change tack and turn our attention to the graded character of the simple modules for the rational Cherednik algebra of type $\mathbf A$. As explained above, we use work of Rouquier, Leclerc-Thibon and Varagnolo-Vasserot to calculate the graded character of these modules. In (\ref{eq:definitionGpoly}), we introduce a class of rational functions $\Gpoly{k}{n}(\lambda,\nu;t)$, defined in terms of Littlewood-Richardson coefficients and $(q,t)$-Kostka polynomials, and show that these functions give the graded $\s_m$-character of a large class of simple modules, Proposition \ref{prop:gpoly}. In this section we also describe the Calaque-Enriquez-Etingof functor which relates equivariant $D$-modules supported on the nilpotent cone to simple modules for the rational Cherednik algebra.

The Calaque-Enriquez-Etingof functor allows us to interpret the
character formulas given in section \ref{sec:rationalCherednik} for
simple modules of the rational Cherednik algebra as the characters of
certain equivariant $D$-modules on the nilpotent cone. As noted above,
the graded character of $\rr$ is closely related to the character of the
\hc module $\mm$. This module is also, via 
a $D$-module  interpretation of the Springer correspondence,
 closely related to the simple, equivariant 
$D$-modules supported on the nilpotent cone, cf. \cite{HottaKashiwara}. 
Therefore, in section
\ref{sec:Dmod}, we compare the characters of these simple $D$-modules
derived from the character formulae of $\rr$ given in section
\ref{sec:polygraph} with the formulae given in terms of the rational
functions that were introduced in section
\ref{sec:rationalCherednik}. This produces (Theorem
\ref{thm:identities}) some interesting and rather mysterious
identities. Another consequence of this comparison is that one can
define a filtration on a large class of simple modules for the rational
Cherednik algebra such that the associated graded object is bigraded and
the bigraded $\s_m$-character, Proposition \ref{prop:rcasimple}, can be
expressed in terms of transformed Macdonald polynomials.\\

Results regarding the torus orbits of principal nilpotent pairs for $\mathfrak{gl}_n$ are given by E. Zoque in the appendix.

\section{Bigraded $G$-character of $\rr$}\label{sec:bigraded}

\subsection{}\label{sec:groupactions} Throughout the paper we take $G$ to be a connected complex reductive group with Lie algebra $\g$. Fix $T\sset G$, a maximal torus, and let $\t=\Lie T$ be the corresponding Cartan subalgebra of $\g$. Let $W$ be the Weyl group associated to $T \subset G$.

Put $\gg:=\g\times\g$. The {\em commuting scheme} $\zz$, of the Lie
algebra $\g$, is defined as the {\em scheme-theoretic} zero fiber of the
commutator map $\kap:\ \gg\to\g,$ $(x,y)\mto[x,y]$. Set-theoretically,
one has $\zz=\{(x,y)\in\gg\mid [x,y]=0\}$. The group $G$ acts on $\g$
via the adjoint action $G\ni g: x\mto \Ad g(x)$ and diagonally on
$\gg$. This makes $\zz$ a closed $G$-stable subscheme of $\gg$. We put
$\tt := \t \times \t$ and let the Weyl group $W$ act
diagonally. Restriction of polynomial functions on $\gg$ to $\tt$ gives rise to a map of algebras $\res : \C[\zz]^G \rightarrow \C[\tt]^ W$ (c.f. \cite[(1.3.1)]{Iso}). 

The \textit{isospectral commuting variety} is defined to be the reduced, closed subvariety
$$
\xx = \{ (x,t) \in \zz \times \tt \en|\en P(x) = (\res P)(t), \ \forall P \in \C[\zz] \}
$$
of $\zz \times \tt$. It is shown in \cite[Theorem 1.3.4]{Iso} that the
normalization $\xx_{\norm}$ of $\xx$ is a Cohen-Macaulay, Gorenstein
variety with trivial canonical bundle. A consequence of this is that
$\zz_{\norm}$ is also Cohen-Macaulay.

 Projection onto the first factor defines a map $\xx \rightarrow
 \zz$. This lifts (\cite[\S 1.4]{Iso}) to a finite morphism $p_{\norm} :
 \xx_{\norm} \rightarrow \zz_{\norm}$.
We define $\rr := (p_\norm)_{*} \ \mathcal{O}_{\xx_{\norm}}$,
a coherent sheaf on $\zz_{\norm}$.

The group $\Cs$ acts on $\g \times \t$ by dilations. Therefore, there is
also an action of $\cc$ on $\gg \times \tt = (\g \times \t) \times (\g
\times \t)$ by dilations. Since the action of $\cc$ commutes with the
action of $G \times W$ on $\gg \times \tt$, it is a $G \times W \times
\cc$-variety. Fix $H := G \times \cc$. The variety $\xx$ is $H \times
W$-stable and, as noted in \cite[\S 1.4]{Iso}, this induces an action of
$H \times W$ on the isospectral variety $\xx_{\norm}$.
 The morphism $p_\norm$ is $G \times W \times \cc$-equivariant,
 therefore $\rr$ is a $H \times W$-equivariant coherent sheaf on $\zz_{\norm}$. 

\subsection{An analogue of the Grothendieck-Springer resolution}\label{tx}
Let $\bb$ be the flag variety, the variety of all Borel subalgebras $\b\sset\g$. The following variety was introduced by the second author in \cite[\S 3.1]{Iso}: 
$$
\tgg:=\{(\b,x,y)\in\bb\times\g\times\g\mid
x,y\in\b\}.
$$
Projection onto the first factor makes $\tgg$ a sub vector bundle of the trivial vector bundle $\bb\times\gg\to\bb$. Let $\mathfrak{q} : \tgg \rightarrow \bb$ and $\mmu:\ \tgg\to \gg,\ (\b,x,y)\mto(x,y)$ denote the $G$-equivariant projections on to the first, respectively the second and third factors. Let the group $\Cs$ act by dilations on each Borel $\b$. Then there is an action of the group $\cc$ on the fibers $\b \times \b$ of $\mathfrak{q}$. This makes $\tgg$ a $H$-variety and the maps $\mathfrak{q}$ and $\mmu$ are $H$-equivariant.  

Let $\T$ denote the tangent  bundle on $\bb$, its fiber at the point
$\b$ is $\g / \b \simeq [\b,\b]^*$. For each $n \ge 0$, let $\aa_n :=
\wedge^n \mathfrak{q}^* \T$, a vector bundle on $\tgg$. By letting $\cc$
act by dilations on the fibers of $\T$, each $\aa_n$ is naturally a
$H$-equivariant sheaf. As explained in \cite[\S 3.4]{Iso}, the
commutator map can be used to define a differential $\partial_\bullet :
\aa_\bullet \rightarrow \aa_{\bullet - 1}$ so that $\aa := \bigoplus_{n
\ge 0} \aa_n$ is a sheaf of coherent DG
$\mathcal{O}_{\tgg}$-algebras. Let $u$ denote the composite map
$\zz_{\norm} \rightarrow \zz \hookrightarrow \gg$. Then one of the main
results of \cite{Iso}, Corollary 4.5.3, says that 
there is an
isomorphism of $H$-equivariant $\mathcal{O}_{\gg}$-modules
\beq{eq:isomainiso}
\mathscr{H}^0(R \mmu_* \aa) \simeq u_* \rr,
\eeq
and $\mathscr{H}^k(R \mmu_* \aa) = 0$ for all $k \neq 0$. 
In other words, in $\dcoh(\gg)$,
the bounded derived category of coherent sheaves,
one has $R \mmu_* \aa\simeq u_* \rr$. This implies that we have an equality 
\beq{eq:isomainisoK}
\mmu_* [\aa] = u_* [\rr]
\eeq
in the Grothendieck group $K^H(\gg)$ of $H$-equivariant coherent $\mathcal{O}_{\gg}$-modules.

\subsection{Degenerate Macdonald polynomials}\label{sub:degenMac}

The representation ring of $G$, respectively $W$, will be denoted $R(G)$, respectively $R(W)$. Let $R_+\sset \t^*$ denote the set of weights of the adjoint $\t$-action on the vector space $[\b,\b]^* \simeq \g / \b$, and let $\ell(-)$ denote the length function on the Weyl group $W$. Denote by $P$ the lattice of integral weights in $\t^*$ and by $Q$ the root lattice. The set of dominant weights (i.e. those weights $\lambda \in P$ such that $\langle \lambda , \alpha^\vee \rangle \ge 0$ for all $\alpha \in R^+$) will be denoted $P^+$ and we set $Q^+ = Q \cap P^+$. The semi-group in $Q$ generated by the positive roots $R_+$ is denoted $Q_+$. Let $\rho$ denote the half-sum of all positive roots. We denote by $J$ the anti-symmetrization map on the ring $\Z[P]$, which is defined by $J(e^\lambda) := \sum_{w \in W} (-1)^{\ell(w)} e^{w(\lambda)}$. For $\mu \in P^+$, the class in $R(G)$ of the irreducible, finite dimensional $G$-module with highest weight $\mu$ will be denoted $s_\mu(z)$. We identify $s_\mu(z)$ with the $T$-character $J(e^{\mu + \rho}) / J(e^\rho)$ of $V_\mu$. All identities that we present will be elements in the ring $\Lambda := \Q (q,t) \o_{\Z} R(G) \o_{\Z} R(W)$. 

For any weight $\lambda \in P$ we define, as in \cite{StembridgeMacKernel}\footnote{we will only consider the case $\omega = \rho$ and omit it from the notation.}, the degenerate Macdonald polynomial to be
$$
P_\lambda(t) = \sum_{w \in W} w \left( e^\lambda  \prod_{\alpha \in  R^+} \frac{(1 - te^{-\alpha})}{(1 - e^{-\alpha})} \right) = \frac{J \left( e^{\lambda + \rho} \prod_{\alpha \in R^+} (1 - t e^{-\alpha}) \right)}{J(e^\rho)},
$$
The degenerate Macdonald polynomial is symmetric, therefore there exist polynomials $\acoe_{\lambda,\mu}(t)$ such that 
\beq{eq:degnexpand}
P_\lambda(t) = \sum_{\mu \in P^+} \acoe_{\lambda,\mu}(t) \cdot s_\mu(z).
\eeq
For each $\lambda \in P$ we denote by $\lambda^+$ the unique dominant weight in the $W$-orbit of $\lambda$. By \cite[Proposition 2.2]{StembridgeMacKernel}, we have $\acoe_{\lambda,\mu}(t) = 0$ unless $\mu \le \lambda^+$. When $\lambda$ is dominant, $P_\lambda(t)$ is the usual Hall-Littlewood-Macdonald polynomial and $\acoe_{\lambda,\mu}(t) = K_{\lambda,\mu}(t)$, the Kostka-Foulkes polynomial. For $\lambda \in P$, define
$$
\widetilde{P}_{\lambda}(t) = \frac{J \left( e^{\lambda + \rho} \prod_{\alpha \in R^+} (1 - t e^{\alpha}) \right)}{J(e^\rho)}.
$$

\begin{lem}\label{lem:tildeP}
For any $\lambda \in P$ we have
$$
\widetilde{P}_{\lambda}(t) = \sum_{\mu \le (\lambda + 2 \rho)^+} \widetilde{\acoe}_{\lambda,\mu}(t) \cdot s_{\mu}(z)
$$
where $\widetilde{\acoe}_{\lambda,\mu}(t) := (-t)^\gss \acoe_{\lambda + 2 \rho,\mu}(t^{-1})$ and $\gss = |R^+|$. 
\end{lem}

\proof
This follows from the identity
$$
\frac{J \left( e^{\lambda + \rho} \prod_{\alpha \in R^+} ( 1 - t e^{\alpha}) \right)}{J(e^\rho)} = (-t)^\gss P_{\lambda + 2 \rho}(t^{-1}). \eqno\Box
$$

\subsection{Bigraded character formula}\label{sec:qtgrading}
Given a reductive group $K$, a $K$-scheme $X$ and a $K$-equivariant coherent sheaf $\mathcal F$ on $X$, we write $\cchi^K(\mathcal F)$ for its equivariant Euler characteristic, a class in the Grothendieck group of rational $K$-modules. We identify the Grothendieck group of rational $\Cs\times\Cs$-modules with the Laurent polynomial ring $\Q [q^{\pm 1}, t^{\pm 1}]$. 

We define a $(q,t)$-analogue $\biP( - ; q,t) : P \longrightarrow \Q [q, t]$ of Kostant's partition function by the generating function
$$
\prod_{\alpha \in R_+} \frac{1}{(1-q e^{\alpha})(1-t e^{\alpha})} = \sum_{\lambda \in Q_+} \biP(\lambda;q,t) e^\lambda.
$$
Each isotypic component of $\C[\gg]$ with respect to the action of $\cc$ is finite dimensional, therefore $\cchi^H(\mathcal{O}_{\gg})$ is a well-defined element in $\Lambda$. Since $\rr$ may be considered as a coherent sheaf on $\gg$, $\cchi^H(\rr)$ is also well-defined as an element in $\Lambda$. The isomorphism (\ref{eq:isomainiso}) allows us to calculate the Euler characteristic $\cchi^{H}(\rr)$ of $\rr$.

\begin{thm}\label{character_2}
The bigraded $G$-character of the global sections of the sheaf $\rr$ is given by the formula 
\beq{eq:biHesselink}
\cchi^{H}(\rr)= \frac{1}{( 1 - q)^{\mathbf r} (1 - t)^{\mathbf r}} \sum_{\stackrel{\lambda \in Q_+;}{\mu \le (\lambda + 2 \rho)^+}} \biP(\lambda;q,t) \cdot \widetilde{\acoe}_{\lambda,\mu}(q t) \cdot s_\mu(z),
\eeq
where ${\mathbf r}$ is the rank of $G$.
\end{thm}

The remainder of this section is devoted to the proof of Theorem \ref{character_2}. We first calculate the Euler characteristic $\cchi^{T \times \cc}(\rr)$ of $\rr$.

\begin{prop}\label{character_one} 
The bigraded $T$-character of the global sections of the sheaf $\rr$ is given by the formula
$$
\cchi^{T\times \Cs\times\Cs} (\rr )= \frac{1}{( 1 - q)^{\mathbf r} (1 - t)^{\mathbf r}} \sum_{w\in W} w \left(  \prod_{\alpha \in R_+} \frac{(1-q t e^{\alpha})}{(1-q e^{\alpha})(1-t e^{\alpha})(1-e^{-\alpha})} \right).
$$
\end{prop}

\begin{proof}
For any locally finite representation $E$, of a Borel subgroup $B\sset G$,  let $\underline{E}$ denote the corresponding induced $G$-equivariant vector bundle on $\bb=G/B$.

The terms of the complex ${\mathfrak q}_\idot\aa^\hdot$ are the quasi-coherent vector bundles ${\mathfrak q}_\idot\aa^n$, whose fiber at $\b \in \bb$ is $\sym \b^* \o \sym \b^* \o (\wedge^n [\b,\b]^*)$. Therefore we get the equation:
\beq{above}
\cchi^{T\times \C^\times\times\C^\times}({\mathfrak q}_\idot \aa^\hdot)=
\sum_{k,m,n\geq 0}\ (-qt)^n q^{k} t^{m}\cdot \cchi^T\big( \sym^k\underline{\b}^*\o\sym^m\underline{\b}^*\o(\wedge^n 
\underline{[\b,\b]}^*)\big).
\eeq

Thanks to the equality (\ref{eq:isomainisoK}), the left hand side of equation \eqref{above} is equal to the bigraded character of the $T$-module $\rr$. To obtain the formula of the proposition, one computes the right hand side of equation \eqref{above} using the Atiyah-Bott fixed point formula for $G$-equivariant vector bundles on the flag variety, as explained in \cite[\S 6.1.16]{CG}.
\end{proof}

\begin{proof}(of Theorem \ref{character_2})
Since $\prod_{\alpha \in R_+} \frac{1}{( 1- e^{-\alpha})} = e^\rho J(e^\rho)^{-1}$ and $J(e^\rho)$ is skew symmetric,
$$
\cchi^{T\times \Cs\times\Cs}(\rr)= \frac{1}{( 1 - q)^{\mathbf r} (1 - t)^{\mathbf r} J(e^\rho)} \sum_{w\in W} (-1)^{\ell(w)} w \left( e^\rho \prod_{\alpha \in R_+}\frac{(1-q t e^{\alpha})}{(1-q e^{\alpha})(1-t e^{\alpha})} \right).
$$
By definition,
$$
 \prod_{\alpha \in R_+}
\frac{1}{(1-q e^{\alpha})(1-t e^{\alpha})} = \sum_{\lambda \in Q_+} \biP(\lambda;q,t) e^\lambda
$$
which means that 
$$
\cchi^{T\times \cc} (\rr)= \frac{1}{( 1 - q)^{\mathbf r} (1 - t)^{\mathbf r} J(e^\rho)} \sum_{\lambda \in Q_+} \biP(\lambda;q,t) J \left( e^{\rho + \lambda} \prod_{\alpha \in R^+} ( 1 - qt e^{\alpha}) \right).
$$
Then equation (\ref{eq:biHesselink}) follows from the formula for $\widetilde{P}_{\lambda}(qt)$ given in Lemma \ref{lem:tildeP}.
\end{proof}

\section{Principal Nilpotent Pairs}\label{sec:pnp}
When restricted to the smooth locus of the commuting variety, the isospectral commuting variety $\rr$ is a bigraded, $G \times W$-equivariant vector bundle. In this section we use the theory of principal nilpotent pairs to study the bigraded character of the fiber of this bundle at fixed points. The main result of this section is the character formula (\ref{eq:pnp1}) of Theorem \ref{thm:pnpcharacter}.

\subsection{Definitions}\label{sec:defns}
For any pair $\mbf{x} = (x_1, x_2) \in \zz$, write $\mathfrak{z}(\mbf{x})$ for the simultaneous centralizer of $x_1$ and $x_2$. By a theorem of Richardson, \cite{RichardsonCommuting}, $G \cdot (\t \times \t)$ is open and dense in $\zz$. This implies that $\dim \mathfrak{z}(\mbf{x}) \ge \rk \g$ for all $\mbf{x} \in \zz$. The pair $(x_1,x_2)$ is said to be \textit{regular} if $\dim \mathfrak{z}(\mbf{x}) = \rk \g$. The set of all regular pairs in $\zz$ is denoted $\zzr$; it is the smooth locus of $\zz$. 

\begin{lem}[Theorem 1.5.2 (i), \cite{Iso}]
The restriction of $\rr$ to $\zzr$ is a locally free sheaf such that each fiber of the corresponding vector bundle affords the regular representation of $W$.
\end{lem}

The paper \cite{PNP} introduced the notion of principal nilpotent pairs in $\zz$: 

\begin{defn}\label{def:pnp}
A pair $\bbe = (e_1,e_2) \in \gg$ is called a \textit{principal nilpotent pair} if the following conditions hold:
\begin{enumerate}
\item $\bbe$ is a regular pair i.e. $\bbe \in \zzr$;
\item For any $(t_1,t_2) \in \cc$, there exists $g = g(t_1,t_2) \in G$ such that $(t_1 \cdot e_1, t_2 \cdot e_2) = (\Ad g(e_1),\Ad g(e_2))$.
\end{enumerate}
\end{defn}
Note that condition (2) of definition \ref{def:pnp} implies that $e_1$ and $e_2$ are nilpotent. It is shown in \cite[Theorem 1.2]{PNP} that one can associate to each \pnp $\bbe \in \zz$ a pair $\mbf{h} = (h_1,h_2) \in \zz$ of semisimple elements of $\g$ such that $[h_i,e_j] = \delta_{i,j} \cdot e_j$ for $i,j = 1,2$ and the adjoint action of $\mbf{h}$ on $\g$ defines a $\Z^2$-grading $\g = \oplus_{p,q \in \Z^2} \ \g_{p,q}$.

\subsection{}
For the remainder of this section we fix a \pnp $\bbe \in \zzr$ and associated semisimple pair $\mbf{h} = (h_1,h_2)$. Without loss of generality, $\mathfrak{z}(\mbf{h}) = \t$. Recall that $H = G \times \cc$ acts on $\g$ by $(g,\alpha,\beta) \cdot x = \Ad(g) \cdot x$ and on $\gg$ by 
$$
(g,\alpha,\beta) \cdot (x,y) =(\alpha\inv\cdot\Ad(g) \cdot x, \beta\inv\cdot\Ad(g) \cdot y).
$$
Let $A := \cc$ and let $\vr: A \to G$ be the homomorphism such that $d\vr:\ \C^2\to\g$ sends $(1,0)$ to $h_1$ and $(0,1)$ to $h_2$. The embedding $A \hookrightarrow H$, $(t_1,t_2) \mapsto (\vr(t_1,t_2),t_1,t_2,)$ defines an action of $A$ on $\g$, respectively on $\gg$. Explicitly, $(t_1,t_2)\bu x=\Ad\vr(t_1,t_2) \cdot x,$ and 
$$
(t_1,t_2)\bu(x,y)=(t_1\inv\cdot\Ad\vr(t_1,t_2) \cdot x, t_2\inv\cdot\Ad\vr(t_1,t_2) \cdot y).
$$
This $A$-action defines a $\Z^2$-grading $\gg=\bplus_{i,j\in\Z}\gg_{i,j}$ such that  
$$
\gg_{i,j} := \g_{i+1,j}\oplus \g_{i,j+1} = \{(x,y)\in\g\oplus\g \mid \Ad\vr(t_1,t_2) \cdot x = t_1^{i+1} t_2^j \cdot x,\ \Ad\vr(t_1,t_2) \cdot y = t_1^it_2^{j+1} y \}.
$$
By definition, $\g^A = \g_{0,0} \ (= \t)$ and $\gg^A = \gg_{0,0} \ (= \g_{1,0}\oplus\g_{0,1})$. Note that we have $e_1\in\g_{1,0}$ and $e_2 \in \g_{0,1}$, so $\bbe \in \gg^A$. The map $\g\to\gg$, $x \mto (\ad e_1(x), \ad e_2(x))$ is $A$-equivariant.

\subsection{}
Since $\bbe$ is fixed by $A$, the fiber $\rr_\bbe$ of $\rr$ at $\bbe$ is bigraded. In order to present a formula for this bigraded vector space we require a few further definitions. 

\begin{defn}  
A Borel subalgebra $\b$  is said to be {\em adapted} (to $e_1,e_2,h_1,h_2$) if it  contains all four elements $e_1,e_2,h_1,h_2$.
\end{defn}

\begin{lem}\label{quad} 
A Borel subalgebra $\b$ is adapted if and only if we have $\oplus_{i,j\geq0}\g_{i,j}\sset\b$.
\end{lem}

\proof 
This follows from the proof of \cite[Theorem 1.13]{PNP} which shows that one has
$$
\bplus_{i,j\geq0}\g_{i,j}\ =\ \sum_{p,q\geq0}\
(\ad e_1)^p(\ad e_2)^q(\t).\eqno\Box
$$

If $\b$ contains $h_1$ and $h_2$ then $\t \subset \b$. Since the set of Borel subalgebras containing $\t$ is naturally in bijection with the elements of $W$, choosing a particular adapted Borel subalgebra $\b_1$ defines a bijection $\b_w \leftrightarrow w$ between the set of all adapted Borel subalgebras of $\g$ and a certain subset $W_{\mathrm{adp}} \subset W$. 

\subsection{Partial slices}

In \cite[\S 7]{PNP} a certain ``partial slice'' to the $G$-orbit in $\zzr$ through $\bbe$ was constructed. Decompose the centralizers $\mathfrak{z}(e_1) = \oplus_{p,q} \mathfrak{z}_{p,q}(e_1)$ and $\mathfrak{z}(e_2) = \oplus_{p,q} \mathfrak{z}_{p,q}(e_2)$ with respect to the action of $A$. For each $p,q$ such that $p \le 0$ and $q \ge 0$, choose a $T$-stable subspace $S_{p,q} \subset \mathfrak{z}_{p,q}(e_2)$ complementary to $\mathrm{Im} (\ad \ e_1 : \mathfrak{z}_{p-1,q}(e_1) \rightarrow \mathfrak{z}_{p,q}(e_1))$ and form the subspace $S_{\nw} := \bigoplus_{p \le 0, q \ge 0} S_{p,q} \subset \g$. Let $\Snw := S_{\nw} \oplus \{ 0 \} \subset \gg$ denote the corresponding subspace in $\gg$. Similarly, by considering all $p \ge 0$ and $q \le 0$, one defines $\Sse := S_{\se} \oplus \{ 0 \} \subset \gg$. The following result is noted in \cite[\S 6]{PNP}. 

\begin{lem}\label{lem:partialslice}
There is an $A$-equivariant isomorphism of vector spaces
$$
T_{\bbe} \zz \simeq \Snw \oplus \Sse \oplus \g / \mathfrak{z}(\bbe).
$$
\end{lem}

\begin{proof}
Recall that $\kappa : \gg \rightarrow \g$ is the commutator map so that $\zz = \kappa^{-1}(0)$. The regular locus $\zzr \subset \zz$ is precisely the set of points where the rank of $d \kappa$ is maximal. Therefore, since $\bbe \in \zzr$, $T_{\bbe} \ \zz = \Ker d_{\bbe} \kappa$. Let $\Ad_{\bbe} : G \rightarrow \gg$ be the map $g \mapsto g \cdot \bbe$. The map $\kappa$ is $G$-equivariant, hence $G \cdot \bbe \subset \zzr$. Differentiating the maps $\Ad_{\bbe}$ and $\kappa$ at $\mbf{1} \in G$ and $\bbe$ respectively gives a three term complex
\beq{eq:threetermcomplex}
\xymatrix{
\g \ar[rr]^<>(0.5){d_{\mbf{1}} \Ad}  && \gg \ar[rr]^<>(0.5){d_{\bbe} \kappa} && \g 
}
\eeq
such that, if $H^{1}$ is the middle cohomology of the above complex, then $T_{\bbe} \ \zz \simeq H^{1} \oplus T_{\bbe} \ (G \cdot \bbe)$. The above complex is (up to sign) precisely the complex \cite[(6.1)]{PNP}. Therefore, in the notation of \cite{PNP}, $H^{1} = H^1(\mathfrak{e},\g)$. Now \cite[Theorem 6.6]{PNP} implies that the natural map $\Snw \oplus \Sse \rightarrow H^1(\mathfrak{e},\g)$ is an isomorphism. Since $T_{\bbe} (G \cdot \bbe) = \g / \mathfrak{z}(\bbe)$, the lemma follows.
\end{proof} 

\begin{lem}\label{lem:Torbitdense}
The free orbit $T \cdot \bbe$ is open in $(\g_{1,0} \oplus \g_{0,1}) \cap \zz$.
\end{lem}

\begin{proof}
Write $Z = (\g_{1,0} \oplus \g_{0,1}) \cap \zz$. Since the component group of $T$ is trivial, the fact that $T \cdot \bbe$ is free follows from Lemma \ref{quad} which implies that $\mathfrak{z}(\bbe) \cap \t = 0$. To show that $T \cdot \bbe$ is open in $Z$ it suffices to show that $T_{\bbe} \ Z = d_{\mbf{1}} \Ad(\t)$. The map $\ad : \g \rightarrow \gg$ is $A$-equivariant. Therefore, since $\g_{1,0} \oplus \g_{0,1} = \gg^A$, Lemma \ref{lem:partialslice} says that 
$$
T_{\bbe} \ Z = (\Snw \oplus \Sse \oplus \g / \mathfrak{z}(\bbe))^A.
$$
Now $\t = \g^A$ and \cite[Theorem 6.6]{PNP} says that $H^1(\mathfrak{e},\g)^A = 0$ which implies that $(\Snw \oplus \Sse \oplus \g / \mathfrak{z}(\bbe))^A = \t$ as required. 
\end{proof}

\begin{rem}
In the case $G = SL_n$, it is shown by E. Zoque in the Appendix that the torus orbit $T \cdot \bbe$ is dense in $(\g_{1,0} \oplus \g_{0,1}) \cap \zz$. 
\end{rem}
\subsection{}
Given a rational $A$-module $V$, denote by $\lambda(V)$ the Euler characteristic of the alternating sum $\sum_{k = 0}^{\dim V} (-1)^k [\wedge^k V] \in K^A(pt) = R(A)$. If $w$ is adapted to $\bbe$ then Lemma \ref{quad} says that the bigrading on $\g$ induces a bigrading $\b_w=\oplus_{i,j}\ (\b_w)_{i,j},$ where $(\b_w)_{i,j}=\b_w\cap \g_{i,j}$. Furthermore, we have $(\b_w)_{0,0}=\g_{0,0}=\t$ and $(\b_w)_{1,1}=\g_{1,1}$. Let $\uf_w$ be the $T \times \cc$-stable complement to $\g_{1,1}$ in $[\b_w,\b_w]$, thus, we have
$$
[\b_w,\b_w]=\bigoplus_{\{i,j\ \mid\  (i,j)\neq(0,0)\}}\ (\b_w)_{i,j},
\quad\text{resp.}\quad
\uf_w  =\bigoplus_{\{i,j\ \mid\  (i,j)\neq(0,0),\ (1,1)\}}\ (\b_w)_{i,j}.
$$
\begin{thm}\label{thm:pnpcharacter}
The bigraded character of $\rr_\be$, the fiber of the vector bundle $\rr$ at the \pnp $\be\in\zz$, is given by the formula
\beq{eq:pnp1}
\cchi^A(\rr_\be)\ =\ \lambda(\Snw^* \oplus \Sse^* \oplus (\g / (\mathfrak{z}(\bbe) \oplus \t))^*) \cdot \sum_{w \in W_{\mathrm{adp}}} \lambda \left( [\b_w, \b_w] \oplus \left( \frac{\b_w \oplus \b_w}{\g_{1,0} \oplus \g_{0,1}} \right)^* \right)^{-1} \cdot \lambda(\uf_w^*).
\eeq
\end{thm}

In order to give meaning to Theorem \ref{thm:pnpcharacter}, we must explain how the torus $A$ acts on the various spaces appearing in formula (\ref{eq:pnp1}). As explained in (\ref{sec:defns}), the action of $A$ comes from its embedding in $\cc \times T$. Therefore we will just remind the reader how $T \times \cc$ acts on these spaces. Recall that in (\ref{sec:defns}) we have defined the action of $T \times \cc$ on $\g$ and $\gg$. The group acts on $\Snw^*$ and $\Sse^*$ as subspaces of $\gg^*$. The action of $T \times \cc$ on $(\g / (\mathfrak{z}(\bbe) \oplus \t))^*$ comes from its action on $\g^*$. The space $( \b_w \oplus \b_w / \g_{1,0} \oplus \g_{0,1} )^*$ is a subspace of $\gg^*$ and $[\b_w, \b_w]$ is a subspace of $\g$. One has to be careful with the action of $T \times \cc$ on $\uf_w^*$ - the space $\uf_w^*$ is a subspace of the fiber of $\mathfrak{q}^* \T$ at $(\b_w,0,0)$. Therefore $T$ acts in the natural way but $\cc$ acts by dilations, hence this is not the action of $T \times \cc$ coming from the fact that $\uf_w \subset \g$.

The remainder of this section is devoted to the proof of Theorem \ref{thm:pnpcharacter}. It is a rather long calculation in equivariant $K$-theory. We begin, following \cite[Chapter 5]{CG}, by describing the basic setup in which we work.

\subsection{Equivariant $K$-theory} 
Given a group $A$ acting on a smooth quasi-projective variety $M$, write $M^A$ for the fixed point set. We will assume that $A$ is abelian and reductive. Then, as shown in \cite[Lemma 5.11.1]{CG}, $M^A$ is also smooth. The Grothendieck group of $A$-equivariant coherent sheaves on $M$ will be denoted $K^A(M)$ and its complexification $K^A_{\C}(M)$. For an $A$-equivariant vector bundle $V$ on $M$, define $\lambda(V) = \sum_{k = 0}^{\rk V} (-1)^k [\wedge^k V] \in K^A(M)$. Fix $a \in A$ to be $M$-regular (that is, $a \in A$ such that $M^a = M^A$) and let $ev_a : K^A(M^A) \rightarrow K_{\C}(M^A)$ be the map $K^A(M^A) = R(A) \o_{\Z} K(M^A) \rightarrow K_{\C}(M^A)$ given by evaluating functions at $a$. Define $\lambda(V)_a := ev_a(\lambda(V)) \in K_{\C}(M^A)$ and (c.f. \cite[(5.11)]{CG}):
$$
res_a (\mathcal{F}) = \lambda(T^*_{M^A} M)_a^{-1} \cdot ev_a(i^* \mathcal{F}) \in K_{\C}(M^A),
$$
where $i : M^A \hookrightarrow M$. Let $K^A(M^A)_a := R(A)_a \o_{R(A)} K^A(M^A)$ denote the localization of $K^A(M^A)$ with respect to the set of functions in $R(A)$ that are non-zero at $a$. Then \cite[Proposition 5.10.3]{CG} says that $\lambda(T^*_{M^A} M)$ is invertible in $K^A(M)_a$. Write 
$$
Res_a (\mathcal{F}) = \lambda(T^*_{M^A} M)^{-1} \cdot i^* \mathcal{F} \in K^A(M^A)_a,
$$
so that $res_a = ev_a \circ Res_a$. Consider the following setup:
$$
\xymatrix{
N^A\ar@{^{(}->}[rr]^<>(0.5){u^A} \ar@{^{(}->}[d]_{\tilde{\epsilon}} && M^A \ar@{^{(}->}[d]_{\epsilon} \\
N \ar@{^{(}->}[rr]^<>(0.5){u} && M 
}
$$
where $u$ is an $A$-equivariant, closed embedding of smooth varieties. When in this situation, we will repeatedly use the following two facts (as explained in \cite[Proposition 5.4.10]{CG}): 
\begin{enumerate}

\item if $\mathcal{F}$ is a sheaf on $M$, whose support is contained in $N$, then $u^* [\mathcal{F}] = \lambda(T^*_N M) \cdot [\mathcal{F} |_N]$.

\item if $\mathcal{F}$ is a locally free sheaf on $M$ then $u^* [\mathcal{F}] = [ u^* \mathcal{F}]$.

\end{enumerate}

Here $T^*_N M$ denotes the conormal bundle of $N$ in $M$.

\begin{lem}\label{lem:calcone}
Let $\mathcal{F}$ be an $A$-equivariant sheaf on $M$ whose support is in $N$. 
\begin{enumerate}
\item In the $K$-group $K^A(N^A)_a$ we have an equality
$$
Res_a (u^* [\mathcal{F}]) = \lambda(T^*_{N^A} N)^{-1} \cdot \lambda(T^*_{M^A} M |_{N^A}) \cdot (u^A)^* Res_{a} ([\mathcal{F}]).
$$
\item In the $K$-group $K_{\C}(N^A)$ we have an equality
$$
res_a (u^* [\mathcal{F}]) = \lambda(T^*_{N^A} N)^{-1}_a \cdot \lambda(T^*_{M^A} M |_{N^A})_a \cdot (u^A)^* res_{a} ([\mathcal{F}]).
$$
\end{enumerate}
\end{lem}

\begin{proof}
As noted above, $u^* [\mathcal{F}] = \lambda(T^*_N M) \cdot [\mathcal{F} |_N]$. Therefore, using the fact that $\tilde{\epsilon}^* \lambda(T^*_N M) = \lambda(T^*_N M \, |_{N^A})$, we have
$$
Res_a (u^* \mathcal{F}) = \lambda(T^*_N M |_{N^A}) \cdot \lambda(T^*_{N^A} N)^{-1} \cdot \tilde{\epsilon}^* [\mathcal{F} |_N].
$$
On the other hand,
$$
(u^A)^* Res_a (\mathcal{F}) = (u^A)^* (\lambda(T^*_{M^A} M)^{-1} \cdot \epsilon^* [\mathcal{F}]) = \lambda(T^*_{M^A} M |_{N^A})^{-1} \cdot (\epsilon \circ u^A)^* [\mathcal{F}].
$$
Since $\epsilon \circ u^A = u \circ \tilde{\epsilon}$,
$$
(\epsilon \circ u^A)^* [\mathcal{F}] = ( u \circ \tilde{\epsilon})^* [\mathcal{F}] = \lambda(T^*_{N} M |_{N^A}) \cdot \tilde{\epsilon}^* [\mathcal{F} |_N],
$$
from which the first equation follows. Applying $ev_a$ to the first equation gives the second.
\end{proof}

\begin{lem}\label{lem:expandlambda}
In $K^A(N^A)_a$ one has 
\beq{eq:expandlambda}
\lambda (T^*_{N} M |_{N^A}) = \lambda (T^*_{N^A} N)^{-1} \cdot \lambda (T^*_{N^A} M^A) \cdot \lambda (T^*_{M^A} M |_{N^A}).
\eeq
\end{lem}

\begin{proof}
The closed embeddings $N^A \hookrightarrow N \hookrightarrow M$ imply that there is a short exact sequence
$$
0 \longrightarrow T_{N^A} N \longrightarrow T_{N^A} M \longrightarrow (T_N M) |_{N^A} \longrightarrow 0
$$
and dually,
$$
0 \longleftarrow T^*_{N^A} N \longleftarrow T^*_{N^A} M \longleftarrow (T^*_{N} M) |_{N^A} \longleftarrow 0.
$$
By \cite[Corollary 5.4.11]{CG}, this implies that $\lambda(T^*_{N^A} M) = \lambda(T^*_{N^A} N) \cdot \lambda (T^*_{N} M |_{N^A})$. Similarly, the closed embeddings $N^A \hookrightarrow M^A \hookrightarrow M$ imply that $\lambda(T^*_{N^A} M) = \lambda(T^*_{N^A} M^A) \cdot \lambda (T^*_{M^A} M |_{N^A})$. Therefore
$$
\lambda (T^*_{N^A} N) \cdot \lambda (T^*_{N} M |_{N^A}) =  \lambda(T^*_{N^A} M^A) \cdot \lambda (T^*_{M^A} M |_{N^A}).
$$
Since $\lambda(T^*_{N^A} N)$ is invertible in $K^A(N^A)_a$, equation (\ref{eq:expandlambda}) follows.
\end{proof}

Given a rational $A$-module $V$, we denote by $\underline{V}$ the $A$-equivariant vector bundle on $M$ defined by the projection $\pi : M \times V \rightarrow M$, where $A$ acts diagonally on $M \times V$. 

\subsection{The DG algebra $\aa$ on $\tgg$}
We now return to the setting of Theorem \ref{thm:pnpcharacter}. Recall from (\ref{tx}) that we have a DG algebra $\aa$ on $\tgg$. This is a complex of $H$-equivariant vector bundles $\aa_n$, whose fiber at $(\b,x,y)$ is $\wedge^n [\b,\b]^*$. Let $[\aa]=\sum_n (-1)^n \cdot [\aa_n]\in K^H(\tgg)$ be the corresponding class in equivariant $K$-theory. Write $\tggr$ for the open subset of $\tgg$ consisting of points $(\b,x,y)$ such that the pair $(x,y)$ is regular. Since the semisimple pair $\mbf{h}$ is regular, $\bb^A = \bb^T$. Then $\tggrA := (\tggr)^A = \bigsqcup_{w \in W} \tggrA_w$ where 
$$
\tggrA_w := \{ \b_w \} \times (\b_w \oplus \b_w) \cap (\g_{1,0} \oplus \g_{0,1})^r.
$$
The pull-back of $\aa_n$ to $\tggrA_w$ is the $T \times A$-equivariant vector bundle $\wedge^n \underline{[\b_w,\b_w]}^*$. By definition,
$$
Res_a( [\aa]) = \sum_{w \in W} \lambda(T^*_{\tggrA_w} \tggr)^{-1} \cdot i_w^* [\aa] = \sum_{w \in W} \lambda(T^*_{\tggrA_w} \tggr)^{-1} \cdot \lambda( \underline{[\b_w,\b_w]}^* ) 
$$
where $i_w : \tggrA_w \hookrightarrow \tggr$ and $\underline{[\b_w,\b_w]}^*$ is the $T \times A$-equivariant vector bundle on $\tggrA_w$ with fibers $[\b_w,\b_w]^*$. Consider the following setup
$$
\xymatrix{
T\cdot\bbe\ \ar@{^{(}->}[r]^<>(0.5){f}& \zzrA\ar@{^{(}->}[rr]^<>(0.5){u^A} \ar@{^{(}->}[dd]^<>(0.5){\tilde{\epsilon}} && \ggregA \ar@{^{(}->}[dd]^{\epsilon} && \tggrA \ar[ll]_<>(0.5){\mmu^A}\ar@{^{(}->}[dd]^<>(0.5){i}\\
 &&&&& \\
&\zzr\ar@{^{(}->}[rr]^<>(0.5){u}&& \ggreg && \tggr \ar[ll]_<>(0.5){\mmu}
}
$$
The idea is to push and pull $Res_a( [\aa])$ all the way back to $T \cdot \bbe$. Recall that we have fixed a $T \times A$-stable complement $\uf_w$ to $\g_{1,1}$ in $[\b_w,\b_w]$ so that $\lambda(\underline{[\b_w,\b_w]}) = \lambda(\underline{\uf}_w) \cdot \lambda(\underline{\g}_{1,1})$ in $K^{T \times A}(T \cdot \bbe)$. We note that $\g_{1,1}$ equals $([\b_w,\b_w])^A$, so it will be important to keep track of $\g_{1,1}$ because $\lambda(\underline{[\b_w,\b_w]})_a = \lambda(\underline{\g}_{1,1})_a = 0$ in $K_{\C}^T(T \cdot \bbe)$, where as $\lambda(\underline{\uf}_w)_a \neq 0$. Recall also that $\mmu : \tgg \rightarrow \gg$ is the projective morphism sending $(\b,x,y)$ to $(x,y)$.

\begin{lem}\label{lem:calculateres}
In the Grothendieck group $K^{T \times A}(T \cdot \bbe)_a$ we have
\beq{eq:groth1}
(u^A \circ f)^* \mmu^A_* Res_a [\aa] = \lambda(\underline{\g}_{1,1}^*) \cdot \sum_{w \in W_{\mathrm{adp}}} \lambda(T^*_{\tgg_w^A} \tgg |_{\ggregA})^{-1} \cdot \lambda(\underline{\uf}_w^*).
\eeq
\end{lem}

\begin{proof}
The restriction of $\mmu^A$ to $\tggrA_w$ is a closed embedding. If $\b_w$ is adapted to $\bbe$ then $\tggrA_w = (\g_{1,0} \oplus \g_{0,1})^r$ and $\mmu^A$ is just the identity on $\tggrA_w$. Pushing forward,
$$
\mmu^A_* Res_a ([\aa]) = \sum_{w \textrm{ adp}} \lambda(T^*_{\tggrA_w} \tggr |_{\ggregA})^{-1} \cdot \lambda(\underline{[\b_w,\b_w]}^*) + \mathrm{Q} \in K^H(\ggregA)_a,
$$
where $\mathrm{Q}$ consists of terms such that $\bbe$ is not in the support (if $\bbe \notin \tggrA_w$ then $\mmu^A$ is not an isomorphism so one must take derived push-forward, but we can ignore the resulting terms). Since the terms in $ \mathrm{Q}$ are the classes of $T \times A$-equivariant sheaves, the fact that $\bbe$ is not in the support of $\mathrm{Q}$ implies that $(f \circ u^A)^* \mathrm{Q} = 0$. Pulling back along $f \circ u^A$ gives the required equation.  
\end{proof} 

\subsection{}
It is also possible to compute an expression for $(u^A \circ f)^* \mmu^A_* Res_a [\aa]$ in terms of $\rr$. Since $R \mmu_* \aa = u_* \rr$ in the derived category $D^b_{coh}(\ggreg)$, we also have $\mmu_* [\aa] = u_* [\rr]$ in $K^{T \times A}(\ggreg)$. 

\begin{lem}\label{lem:swapresmmu}
In the Grothendieck group $K^{T \times A}(T \cdot \bbe)_a$ we have the equality
$$
\lambda(T^*_{T \cdot \bbe} \zzr)^{-1} \cdot \lambda(T^*_{\ggregA} \ggreg \, |_{T \cdot \bbe}) \cdot (u^A \circ f)^* \mmu^A_* Res_a [\aa] = \lambda(T^*_{\zzr} \ggreg \, |_{T \cdot \bbe}) \cdot Res_a  [\rr].
$$
\end{lem}

\begin{proof}
By \cite[Proposition 5.4.10]{CG}
$$
u^*( \mmu_* [\aa]) = u^*u_* [\rr] = \lambda(T^*_{\zzr} \ggreg) \cdot [\rr] \in K^{T \times A}(\zzr).
$$
This implies that 
$$
Res_a \ u^*( \mmu_* [\aa]) = Res_a (\lambda(T^*_{\zzr} \ggreg) \cdot [\rr]) \in K^{T \times A}(\zzrA)_a.
$$
Since $\lambda(T^*_{\zzr} \ggreg)$ is an alternating sum of vector bundles on $\zzr$, 
$$
Res_a (\lambda(T^*_{\zzr} \ggreg) \cdot [\rr]) = \lambda(T^*_{\zzr} \ggreg \, |_{\zzrA}) \cdot Res_a [\rr],
$$
and hence 
\beq{eq:a1}
Res_a \ u^* ( \mmu_* [\aa]) = \lambda(T^*_{\zzr} \ggreg \, |_{\zzrA}) \cdot Res_a [\rr] \in K^{T \times A}(\zzrA)_a.
\eeq
Since $\mmu_* [\aa]$ is the class of a sheaf supported on $\zzr$, Lemma \ref{lem:calcone} with $N = \zzr$ and $M = \ggreg$ implies that
\beq{eq:a2}
Res_a \ u^* ( \mmu_* [\aa]) = \lambda(T^*_{\zzrA} \zzr)^{-1} \cdot \lambda(T^*_{\ggregA} \ggreg \, |_{\zzrA}) \cdot (u^A)^* Res_a \ \mmu_* [\aa].
\eeq
As in the proof of \cite[Theorem 5.11.7]{CG}, we can apply \cite[Proposition 5.3.15]{CG} to conclude that $Res_a \ \mmu_* [\aa] =  \mmu^A_* Res_a \ [\aa]$. Combining equations (\ref{eq:a1}) and (\ref{eq:a2}) produces the required formula.
\end{proof}

\begin{lem}\label{lem:canelequality}
In $K^{T \times A}(T \cdot \bbe)$ we have $\lambda(\underline{\g}_{1,1}^*) = \lambda( T^*_{\zzrA} \ggregA)$.
\end{lem}

\begin{proof}
Since we are applying $\lambda( - )$ to $T \times A$-equivariant vector bundles on $T \cdot \bbe$, it suffices to show that the fibers of these vector bundles at $\bbe$ are isomorphic as $A$-modules. We have the following subsequence of the sequence (\ref{eq:threetermcomplex}) considered in the proof of Lemma \ref{lem:partialslice}:
\beq{eq:exactone}
0 \longrightarrow \t = \g_{0,0} \longrightarrow \g_{1,0} \oplus \g_{0,1} \longrightarrow \g_{1,1} \longrightarrow 0,
\eeq
where the first map is $\ad e_1 \oplus \ad e_2$ and the second is $\ad e_1 - \ad e_2$. This sequence is exact thanks to \cite[Theorem 6.6]{PNP} and \cite[Proposition 1.12,``Weak Lefschetz'']{PNP}. Since $T_{\bbe} \ (T \cdot \bbe) \simeq \t$ and $\gg^A = \g_{1,0} \oplus \g_{0,1}$, we have 
$$
(T_{T \cdot \bbe} \gg^A)_{\bbe} = (T_{\bbe} \gg^A) / (T_{\bbe} \ (T \cdot \bbe)) \simeq (\g_{1,0} \oplus \g_{0,1}) / \g_{0,0} \simeq \g_{1,1},
$$
where the last isomorphism is due to (\ref{eq:exactone}). On the other hand, Lemma \ref{lem:Torbitdense} says that the orbit $T \cdot \bbe$ is open in $(\g_{1,0} \oplus \g_{0,1}) \cap \zzr = \zzrA$. Therefore, we deduce that
$$
(T_{\zzrA} \ggregA)_{\bbe} = (T_{T \cdot \bbe} \gg^A )_{\bbe} \simeq \g_{1,1}.
$$
It follows that $T_{\zzrA} \ggregA = \underline{\g}_{1,1}$ and hence $T_{\zzrA}^* \ggregA = \underline{\g}_{1,1}^*$. 
\end{proof}

Since $K^{T \times A}(T \cdot \bbe) \simeq R(\cc)$ is a domain, we can cancel non-zero terms in equations holding in $K^{T \times A}(T \cdot \bbe)$.  Let $V$ be a rational $T \times A$-module and $\underline{V}$ the corresponding $T \times A$-equivariant vector bundle on $T \cdot \bbe$. Then $\lambda(\underline{V}) \neq 0$ in $K^{T \times A}(T \cdot \bbe)$ if and only if the weights of $V$ under the $T \times A$-action are all non-zero. 

\begin{prop}
In $K_{\C}(pt) \simeq \C$, 
\beq{eq:pnp2}
\Tr(a; \rr_\bbe)\ =\ \lambda((T^*_{\zzrA} \zzr)_{\bbe})_a \cdot \sum_{w \in W_{\mathrm{adp}}} \lambda((T^*_{\tggrA_w} \tggr )_{\bbe})^{-1}_a \cdot \lambda(\uf_w^*)_a
\eeq
\end{prop}

\begin{proof}
In the case $M = \ggreg$ and $N = \zzr$, Lemma \ref{lem:expandlambda} says that the equality
\beq{eq:gettingthere}
\lambda (T^*_{\zzr} \ggreg |_{\zzrA}) = \lambda (T^*_{\zzrA} \zzr)^{-1} \cdot \lambda (T^*_{\zzrA} \ggregA) \cdot \lambda (T^*_{\ggregA} \ggreg |_{\zzrA})
\eeq
holds in $K^{T \times A}(\zzrA)_a$. Combining equation (\ref{eq:gettingthere}) with Lemmata \ref{lem:calculateres}, \ref{lem:canelequality} and \ref{lem:swapresmmu}, together with the fact that $\lambda(\underline{\g}_{1,1})$ and $\lambda(T^*_{\ggregA} \ggreg \, |_{T \cdot \bbe})$ are not zero-divisors in $K^{T \times A}(T \cdot \bbe)_a$ produces
\beq{eq:almostthere}
Res_a [\rr] = \sum_{w \in W_{\mathrm{adp}}} \lambda(T^*_{\tggrA_w} \tggr |_{T \cdot \bbe})^{-1} \cdot \lambda(\underline{\uf}_w^*)
\eeq
in $K^{T \times A}(T \cdot \bbe)_a$. Recall that 
$$
Res_a [\rr] = \lambda((T^*_{\zzrA} \zzr |_{T \cdot \bbe})^{-1} \cdot \tilde{\epsilon}^{*} [\rr],
$$
and hence 
\beq{eq:ran}
res_a [\rr] = \lambda((T^*_{\zzrA} \zzr |_{T \cdot \bbe})^{-1}_a \cdot ev_a(\tilde{\epsilon}^{*} [\rr]).
\eeq
Since $T$ acts freely on $T \cdot \bbe$, we have $K^T(T \cdot \bbe) \simeq K(\bbe)$. Let $\pi^T_* : K^T_{\C}(T \cdot \bbe) \rightarrow K_{\C}(\bbe)$ denote the isomorphism of descent (c.f. \cite[(5.2.15)]{CG}). Applying $\pi_*^T$ to equations (\ref{eq:ran}) and (\ref{eq:almostthere}) produces (\ref{eq:pnp2}). 
\end{proof}

\subsection{Cotangent Spaces}

To complete the proof of Theorem \ref{thm:pnpcharacter} we just need to show that equation (\ref{eq:pnp1}) is equivalent to equation (\ref{eq:pnp2}). This follows from:

\begin{lem}
As $A$-equivariant vector spaces,
$$
(T^*_{\zzrA} \zzr)_{\bbe} \simeq \Snw^* \oplus \Sse^* \oplus (\g / (\mathfrak{z}(\bbe) \oplus \t))^*,
$$
and, for $w$ adapted to $\bbe$ and $v := (\b_w, \bbe) \in \tggrA_w$,
$$
(T^*_{\tggrA_w} \tggr)_v \simeq [\b_w, \b_w] \oplus \left( \frac{\b_w \oplus \b_w}{\g_{1,0} \oplus \g_{0,1}} \right)^*.
$$
\end{lem}

\proof
At $\bbe$, $T_{\bbe} \zzr \simeq \Snw \oplus \Sse \oplus \g / \mathfrak{z}(\bbe)$. The fact that $T$ acts freely at $\bbe$ and the corresponding orbit is open in $\zzrA$ implies that $T_{\bbe} \zzr / T_{\bbe} \zzrA \simeq \Snw \oplus \Sse \oplus \g / (\mathfrak{z}(\bbe) \oplus \t)$ so that 
$$
(T^*_{\zzrA} \zzr)_{\bbe} \simeq \Snw^* \oplus \Sse^* \oplus (\g / (\mathfrak{z}(\bbe) \oplus \t))^*.
$$
At $v$, $T_v \tggr = \g / \b_w \oplus \b_w \oplus \b_w$ and $T_v \tggrA_w = \g_{1,0} \oplus \g_{0,1}$ so 
$$
(T^*_{\tggrA_w} \tggr)_v \simeq [\b_w, \b_w] \oplus \left( \frac{\b_w
\oplus \b_w}{\g_{1,0} \oplus \g_{0,1}} \right)^*.
\eqno\Box
$$

\section{Principal nilpotent pairs for $\mathfrak{gl}_n$}\label{sec:glnpnp}

In this section, we focus on the case $G = GL_n$ and hence $W = \s_n$, the symmetric group. The aim of this section is to give an explicit combinatorial expression, Theorem \ref{thm:pnpcharactergl}, for formula (\ref{eq:pnp1}). In the second part of this section we show that this combinatorial expression is equivalent to a formula of Garsia and Haiman. We begin by recalling some standard combinatorics related to the representation theory of the groups $GL_n$ and $\s_n$. 

\subsection{Partitions}\label{sec:partitions}
When $G = GL_n$, the ring $\Lambda$ is $\Q(q,t) [z_1^{\pm 1}, \ds,z_n^{\pm 1}]^{\s_n} \o_{\Z} R(\s_n)$ and the character $s_{\mu}(z)$ is the Schur polynomial labeled by $\mu$. The standard inner product on $\Lambda$, with respect to which the Schur polynomials form an orthonormal basis, will be denoted $\langle - , - \rangle$. The complete symmetric function labeled by the partition $\mu$ is denoted $h_{\mu}(z)$. 

Let $\mu$ be a partition of $n$ of length $\ell(\mu)$ and denote by $n(\mu) = \sum_{i = 1}^{\ell(\mu)} (i-1)\mu_i$ the partition statistic. The \textit{Young diagram} of $\mu$ is defined to be the subset $Y_\mu := \{ (i,j) \in \Z^2 \, |\, 0 \le j \le \ell(\mu) - 1, \, 0 \le i \le \mu_j - 1 \}$ of $\Z^2$. Each box in the diagram is called a \textit{node}. The Young diagram should be visualized as a stack of boxes, justified to the left; for example the partition $(4,3,1)$ is: 
$$
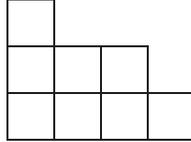
\Yboxdim18pt
\young(\emptybox::::,\emptybox\emptybox\emptybox::,\emptybox\emptybox\emptybox\emptybox:)
$$
This convention is chosen to agree with \cite{HaimanSurvey}. We put
\beq{bnu}
B_\mu(q,t) := \sum_{(r,s) \in Y_{\mu}} q^r t^s.
\eeq
 We adopt the convention that $B_{\emptyset}(q,t) = 0$, where $\emptyset$ is the
empty partition. 

For $x \in Y_\mu$, the \textit{arm} $a(x)$ of $x$ is defined to be the number
of boxes strictly to the right of $x$ and the \textit{leg} $l(x)$ of $x$
is the number of boxes strictly above $x$. We denote by $h(x)$ the
\textit{hook length} of $x$, which is defined to be $a(x) + l(x) +
1$. For instance, the hook length of $(1,0)$ in the above Young diagram
is $4$. The hook polynomial is defined to be 

\begin{displaymath}
H_{\mu}(t) = \prod_{x \in Y_\mu} (1 - t^{h(x)}).
\end{displaymath}
The dominance ordering on partitions will be denoted $\ge$. A \textit{rim-hook} of the partition $\mu$ is a connected skew partition $\mu / \nu$, for some $\nu \subset \mu$, such that $\mu / \nu$ contains no sub-diagram of type $(2,2)$. An \textit{$r$-rim-hook} is a rim-hook of size $r$. The \textit{$r$-core} of a partition is the partition obtained by removing as many $r$-rim-hooks as possible. For any given partition, the resulting $r$-core does not depend on the choice of $r$-rim-hooks removed (see \cite[Theorem 2.7.16]{JK}). For example, the $4$-core of $(4,3,3,1)$ is $(2,1)$.

A tableau $\sigma$ of shape $\mu$ is a filling of the Young diagram $Y_\mu$ with the numbers $1,\ds, n$, each number appearing exactly once. The tableau $\sigma$ is said to be standard if $\sigma(x) < \sigma(y)$ for all $y \in Y_\mu$ that are above or to the right of $x$. The set of all tableaux of shape $\mu$ is written $\YT (\mu)$ and the set of standard tableaux of shape $\mu$ is $\SYT (\mu)$. Let $\sigma_1$ the standard tableau of shape $\mu$ that is given by placing $1$ in $(0,0)$, $2$ in $(1,0)$, filling across and then beginning the next row at the left and working across.

\subsection{Plethysums}\label{sec:plethysm}

We recall the definition of plethystic substitutions as given in \cite{Explicitplethystic}, see also \cite[\S 3.3]{HaimanSurvey}. Let $\tilde{\Lambda}_{\Fqt}$ denote the algebra of symmetric functions in $z_1,z_2,\ds$ over the field $\Fqt = \Q (q,t)$. It is freely generated by the power-sum polynomials $p_k := z_1^k + z_2^k + \cdots$. If $E = E[r_1,r_2, \ds]$ is a formal Laurent polynomial in indeterminates $r_1,r_2, \ds$, which may include the parameters $t$ and $q$, then $p_k[E]$ is defined to be the formal Laurent polynomial $E[r_1^k,r_2^k,\ds]$. Since any $f \in \tilde{\Lambda}_{\Fqt}$ is a polynomial $g(p_1,p_2, \ds )$ in the power-sum polynomials, we define
$$
f[E] = g(p_1,p_2, \ds ) |_{p_k \mapsto p_k[E]}.
$$
The operation $f \mapsto f[E]$ is called the \textit{plethystic substitution} of $E$ into $f$. When $E = Z := z_1 + z_2 + \ds$, we have $p_k[Z] = p_k$ and hence $f[Z] = f$. 

Since we will only use the plethystic substitution in two specific situations, we describe more explicitly what it entails in these situations. Firstly, we have 
$$
E = Z/(1 - t) = r_1 + r_2 + \ds, 
$$
where $r_1,r_2, \ds = z_1,z_2, \ds, tz_1, t z_2, \ds, t^2 z_1, t^2 z_2, \ds$ and
$$
E = Z(1 - t) = r_1 + r_2 + \ds - r_1' - r_2' - \ds, 
$$
where $r_1,r_2, \ds = z_1, z_2, \ds$ and $r_1',r_2', \ds = t z_1, t z_2, \ds$. This implies that  
$$
p_k \left[ \frac{Z}{(1 - t)} \right] = \frac{1}{(1-t^k)} p_k(z), \quad \textrm{resp.} \quad p_k[ Z(1 - t)] = (1 - t^k) p_k(z).
$$
Therefore the operations $f \mapsto f[Z/(1 - t)]$ and $f \mapsto f[Z(1 - t)]$ are inverse to each other. In the second situation we take $E \in \Z[q^{\pm 1}, t^{\pm 1}]$ so that $E = E[r_1,r_2, \ds]$ with $\{r_1,r_2, \ds \} = \{ q^i t^j \ | \ (i,j) \in \Z^2 \}$. For instance, if 
$$
E = 2 t - 3 q t^{-2} + 7 q^3 - 5 \quad \textrm{ then } \quad p_k[E] = 2 t^k - 3 q^k t^{-2k} + 7 q^{3 k} - 5.
$$

\subsection{}

We will also use the $\Omega$ notation. Define $\Omega [ Z ] := \prod_i ( 1 - z_i)$ so that 
$$
\Omega = \exp \left( - \sum_{k = 1}^\infty p_k / k \right).
$$
We remark that our definition of $\Omega$ is \textit{different} from the one given in \cite[\S 3.3]{HaimanSurvey}, where $\Omega [ Z ] := \prod_i \frac{1}{ 1 - z_i}$. As noted in \textit{loc. cit.}, the equalities $p_k[A + B] = p_k[A] + p_k[B]$ and $p_k[-A] = - p_k[A]$ imply that
$$
\Omega[A + B] = \Omega[A] \cdot \Omega[B], \quad \Omega[-A] = 1 / \Omega[A].
$$
The operator $\Omega$ will only be used in one specific type of situation, which we now describe. As above, take $E \in \Z[q^{\pm 1}, t^{\pm 1}]$. Then 
$$
E = \sum_{(i,j) \in \Z^2} a_{i,j} q^i t^j, \quad \textrm{ and } \quad \Omega[E] = \prod_{(i,j) \in \Z^2} (1 - q^i t^j)^{a_{i,j}},
$$ 
where all but finitely many $a_{i,j}$ equal zero. However, if $a_{0,0} < 0$ then $\Omega[E]$ is undefined and if $a_{0,0} > 0$ then $\Omega[E] = 0$. Therefore we define $\Omega [E]^0 := \Omega [ E - a_{0,0}]$ with the convention that $\Omega [0]^0 = 1$. An example:
$$
\Omega [ q^2t^3 - 3 t^{-2} + 2 + 2 q t^{-3} ]^0 = (1 - q^2t^3)(1 - t^{-2})^{-3}(1 - q t^{-3})^2.
$$

\subsection{}
The set of \pnps for $\g = \mathfrak{gl}_n$ are naturally labeled by partitions of $n$. For a given partition $\mu$, one should think of the pair $(e_1,e_2)$ labeled by $\mu$ as a pair of operators acting on the boxes of $Y_{\mu}$. The operator $e_1$ moves each box one to the right and $e_2$ moves each box up by one. We use the standard tableau $\sigma_1$ defined above to realize $e_1$ and $e_2$ as matrices in $\g$. Each node of the tableau $\sigma_1$ corresponds to an element of the standard basis $\{ v_1, \ds, , v_n \}$ of the vectorial representation $V$ of $\g$. Therefore if $v_i$ lies in a given box of $Y_{\mu}$ and $v_j$ is in the box to its right then $e_1 \cdot v_i = v_j$. This gives our matrix corresponding to $e_1$. In a similar way we get the matrix for $e_2$. For instance, if $x$ denotes the principal nilpotent operator with ones just below the diagonal and zeros elsewhere, so that $x \cdot v_i = v_{i+1}$ for all $1 \le i \le n-1$, then $\mu = (n)$ labels the pair $(x,0)$ and $(1^n)$ labels the pair $(0,x)$.

In section \ref{sec:pnp} we fixed a \pnp $\bbe$ with associated semi-simple pair $\mbf{h}$ and corresponding group $A$ so that the expression in Theorem \ref{thm:pnpcharacter} is a sum over all adapted Borel subalgebras of $\g$. Here we take a different approach and fix $\b_1$ to be the Borel subalgebra of lower triangular matrices in $\gl_n$. Then $R_+ = \{ \alpha_{i,j} \ | \ 1 \le i < j \le n \}$ is the set of positive roots corresponding to the weights of $T$ in $[\b_1, \b_1]^*$. For each $w \in \s_n$, the Borel $\b_w := w \cdot \b_1$ is adapted to $\bbe$ if and only if $\b_1$ is adapted to $w^{-1} \cdot \bbe$ i.e. if and only if $\b_1$ contains the elements $w^{-1} \cdot e_1$, $w^{-1} \cdot e_2$, $w^{-1} \cdot h_1$, and $w^{-1} \cdot h_2$.

\begin{lem}\label{tab} 
Let $\bbe$ be the \pnp labeled by $\mu$. Then, the adapted Borel subalgebras are parametrized by the standard tableaux of shape $\mu$.
\end{lem}

\begin{proof}
The symmetric group $\s_n$ acts freely and transitively on $\YT (\mu)$. Therefore we can identify the set $\SYT (\mu)$ of standard tableaux with a certain subset of $\s_n$, $w \leftrightarrow w \cdot \sigma_1 =: \sigma_w$. We just need to show that this subset of $\s_n$ is precisely $(\s_n)_{\mathrm{adp}}$. Let $D$ be some operator acting by moving the boxes of $Y_\mu$. Each tableau $\sigma_w$ gives a realization of $D$ as some linear operator on $V$ and one can see that $D$ is in $\b_1$ if and only if $\sigma_w (D(x)) \le \sigma_w (x)$ for all $x \in Y_\mu$. Since $e_1$ moves things to the right and $e_2$ moves things up we see that these operators belong to $\b_1$ if and only if $\sigma_w$ is a standard tableau.   
\end{proof}

As in \cite[\S 5]{PNP}, each $\sigma \in \SYT (\mu)$ gives us a canonical choice of associated semisimple pairs $(h_1,h_2)$: if we enumerate the nodes $(p,q) \in Y_\mu$ such that $\sigma(a_k,b_k) = k$ then define $h_1 = (a_1, \ds, a_n)$ and $h_2 = (b_1, \ds, b_n)$. The pair $(h_1,h_2)$ is regular semi-simple with $\zf(h_1,h_2) = \mathrm{diag}(\mathfrak{gl}_n) =: \t$. 

\subsection{} The theorem below provides  a purely combinatorial formula for 
the bigraded character of $\rr_{\bbe}$ in the $GL_n$ case.

For a standard tableau $\sigma$, let $c_{\sigma}(i)$ denote the column of $Y_\mu$ containing $i$ and $r_\sigma(i)$ the row of $Y_\mu$ containing $i$ so that $\sigma(c_{\sigma}(i),r_{\sigma}(i)) = i$. A standard tableau $\sigma$ of $\mu$ defines a nested sequence of partitions $\emptyset = \sigma(0) \subset \cdots \subset \sigma(n) = \lambda$.

\begin{thm}\label{thm:pnpcharactergl}
Let $\be\in\zz$ be the \pnp labeled by $\mu$ and, for each standard tableau $\sigma$ of $\mu$, define
$$
\Gamma_{\sigma}(q,t) = \prod_{k = 1}^n \Omega \left[ (1 - q - t + qt) B_{\sigma(k-1)}(q,t) q^{-c_{\sigma}(k)} t^{-r_{\sigma}(k)} \right]^0.
$$
Then the bigraded character of $\rr_\be$ is given by the formula
\beq{eq:pnpgl}
\cchi^A(\rr_\be)\ =\ \frac{\prod_{x \in Y_\mu} (1 - q^{1 + a(x)} t^{- l(x)}) (1 - q^{- a(x)} t^{1 + l(x)})}{(1 - q)^n (1 - t)^n \cdot \Omega[B_\mu(q^{-1},t^{-1})]^0} \cdot \sum_{\sigma \in \SYT (\mu)} \Gamma_{\sigma}(q,t).
\eeq
\end{thm}

\begin{examp}
If $\mu = (2,1)$ then
$$
\cchi^A(\rr_\be) = qt + 2 q + 2 t + 1,
$$
or if $\mu = (3,1)$ then
$$
\cchi^A(\rr_\be) = q^3 t + 3 q^3 + 3 q^2 t + 5 q^2 + 5 q t + 3 q + 3 t + 1.
$$
\end{examp}

\subsection{Proof of Theorem \ref{thm:pnpcharactergl}}
To prove Theorem \ref{thm:pnpcharactergl} we simply need to evaluate the various terms appearing in equation (\ref{eq:pnp1}). Fix a standard tableau $\sigma$ of shape $\mu$. Note that if $V$ is an $A$-module then $\lambda(V) = \Omega[ \cchi^A(V)]$. 

\begin{lem}
In $\Q(q,t)$ we have 
\beq{eq:nwseformula} 
\lambda(\Snw^* \oplus \Sse^*) = \prod_{x \in Y_\mu} (1 - q^{1 + a(x)} t^{- l(x)}) (1 - q^{- a(x)} t^{1 + l(x)}),
\eeq
and 
\beq{eq:denomformula} 
\lambda((\g / (\t \oplus \zf(\bbe))^*) = \Omega \left[ \sum_{(r,s) \neq (p,q) \in Y_\mu} q^{r - p} t^{s - q} - \sum_{(0,0) \neq (r,s) \in Y_\mu} q^{-r} t^{-s} \right].
\eeq
\end{lem}

\begin{proof}
First we show (\ref{eq:nwseformula}). For each $x = (p,q) \in Y_\mu$, let $f_{\mbf{\nu}(p,q)} \in S_{\se}$ be the element defined in \cite[\S 7, Example]{PNP}. The set $\{ f_{\mbf{\nu}(p,q)} \ | \ (p,q) \in Y_{\mu} \}$ is a basis of $S_{\se}$. Then  
$$
[h_1,f_{\mbf{\nu}(p,q)}] = (p_{max} - p) f_{\mbf{\nu}(p,q)}, \quad [h_2,f_{\mbf{\nu}(p,q)}] = (q - q_{max}) f_{\mbf{\nu}(p,q)}.
$$  
Similarly, if $f_{\mbf{\nu}(p,q)} \in S_{nw}$ is the element defined in \cite[\S 7, Example]{PNP}, then  
$$
[h_1,f_{\mbf{\nu}(p,q)}] = (p - p_{max}) f_{\mbf{\nu}(p,q)}, \quad [h_2,f_{\mbf{\nu}(p,q)}] = (q_{max} - q) f_{\mbf{\nu}(p,q)}.
$$  
The above expressions can be rewritten in terms of arms and legs as
$$
[h_1,f_{\mbf{\nu}(x)}] = a(x)  f_{\mbf{\nu}(x)}, \quad [h_2,f_{\mbf{\nu}(x)}] = -l(x) f_{\mbf{\nu}(x)}, \, \forall f_{\mbf{\nu}(x)} \in  S_{se},
$$  
$$
[h_1,f_{\mbf{\nu}(x)}] = -a(x)  f_{\mbf{\nu}(x)}, \quad [h_2,f_{\mbf{\nu}(x)}] = l(x) f_{\mbf{\nu}(x)}, \, \forall f_{\mbf{\nu}(x)} \in  S_{nw}.
$$  
The action of $A$ on $\Snw$ is shifted in comparison to the action of $A$ on $S_{\nw}$ (and similarly for $\Sse$). This accounts for the extra $1$'s in (\ref{eq:nwseformula}).
Now we show (\ref{eq:denomformula}). We have
$$
\lambda((\g / (\t \oplus \zf(\bbe))^*) = \Omega \, [ \cchi^{A}(\g^*) - \cchi^{A}(\t^*) - \cchi^{A}(\zf(\be)^*) ].
$$
Equation (\ref{eq:denomformula} ) follows from 
$$
\cchi^{A}(\g^*) = n + \sum_{\alpha \in R} q^{\alpha(h_1)} t^{\alpha(h_2)} = \sum_{(r,s),(p,q) \in Y_\mu} q^{r - p} t^{s - q};
$$
$$
\cchi^{A}(\t) = n \quad \textrm{ and } \quad \cchi^{A}(\zf(\be)^*) = \sum_{(0,0) \neq (r,s) \in Y_\mu} q^{-r} t^{-s}.
$$
Here we have used the fact, \cite[Theorem 5.6]{PNP}, that $\zf(\bbe)$ has basis $\{ e_1^r e_2^s \, | \, (r,s) \in Y_\mu \backslash \, \{ \, (0,0) \, \} \, \}$. 
\end{proof}

\subsection{} 
Dropping $\sigma$ from the notation, $c(i)$ will denote the column containing $i$, $r(i)$ the row containing $i$ and, for all $i,j \in \{1, \ds, n \}$, we write $c(i,j) := c(i) - c(j)$, $r(i,j) := r(i) - r(j)$. Then 
$$
\langle \alpha_{i,j}, h_1 \rangle = c(i,j), \quad \langle \alpha_{i,j}, h_2 \rangle = r(i,j) \quad \forall \ 1 \le i \neq j \le n.
$$
Hence, for $(p,q) \neq (0,0)$, $\b_{p,q} = \oplus_{\alpha_{i,j}} \g_{\alpha_{i,j}}$, where the sum is over all $i > j$ such that $c(i,j) = p$ and $r(i,j) = q$ (recall that $\b_{0,0} = \t$). We have 
$$
\cchi^A(\b_\sigma^* \oplus \b_\sigma^*) = n q + n t + \sum_{1 \le i < j \le n} q^{1 + c(i,j)} t^{r(i,j)} + q^{c(i,j)} t^{1 + r(i,j)}.
$$
Note that the constant term of $\cchi^A(\b_\sigma^* \oplus \b_\sigma^*)$ is $\cchi^A(\g_{1,0} \oplus \g_{0,1}) = \dim (\g_{1,0} \oplus \g_{0,1})$, so 
$$
\lambda \left( \left(\frac{\b_\sigma \oplus \b_\sigma}{\g_{1,0} \oplus \g_{0,1}} \right)^* \right) = \Omega \left[ n q + n t + \sum_{1 \le i < j \le n} q^{1 + c(i,j)} t^{r(i,j)} + q^{c(i,j)} t^{1 + r(i,j)} \right]^0.
$$
Also,
$$
\lambda([\b_\sigma, \b_\sigma]) = \prod_{\alpha > 0} (1 - q^{-\alpha(h_1)} t^{-\alpha(h_2)}) = \Omega \left[ \sum_{1 \le i < j \le n} q^{c(j,i)} t^{r(j,i)} \right],
$$
and
$$
\lambda(\uf_\sigma^*) = \prod_{\alpha \in J} (1 - q^{\alpha(h_1) + 1} t^{\alpha(h_2) + 1}) = \Omega \left[ \sum_{1 \le i < j \le n} q^{c(i,j) + 1} t^{r(i,j) + 1} \right]^0,
$$
where $J = \{ \alpha \in R^+ \ | \ (\alpha(h_1),\alpha(h_2)) \neq (-1,-1) \}$. For fixed $1 \le k \le n$, one has
$$
\sum_{i < k} q^{c(i,k)} t^{r(i,k)} = q^{-c(k)} t^{-r(k)} B_{\sigma(k-1)}(q,t).
$$ 
Hence, if for $1 \le k \le n$, we define 
$$
A_{\sigma(k)}(q,t) := (qt - q - t) q^{-c(k)} t^{-r(k)} B_{\sigma(k-1)}(q,t) - q - t - q^{c(k)} t^{r(k)} B_{\sigma(k-1)}(q^{-1},t^{-1}),
$$
then 
$$ 
\lambda \left( [\b_\sigma, \b_\sigma] \oplus \left( \frac{\b_\sigma \oplus \b_\sigma}{\g_{1,0} \oplus \g_{0,1}} \right)^* \right)^{-1} \cdot \lambda(\uf_\sigma^*) = \Omega \left[ \sum_{k = 1}^n A_{\sigma(k)}(q,t) \right]^0.
$$
Similarly, for $1 \le k \le n$, we define 
$$
D_{\sigma(k)}(q,t) = q^{-c(k)} t^{-r(k)} B_{\sigma(k-1)}(q,t) + q^{c(k)} t^{r(k)} B_{\sigma(k-1)}(q^{-1},t^{-1}) - q^{-c(k)} t^{-r(k)},
$$ 
so that equation (\ref{eq:denomformula}) shows  
$$
\lambda((\g / (\t \oplus \zf(\be)))^*) = \Omega \left[ \sum_{k = 1}^n D_{\sigma(k)}(q,t) \right]^0,
$$
and 
$$
\lambda((\g / (\t \oplus \zf(\be)))^*) \cdot \left( [\b_\sigma, \b_\sigma] \oplus \left( \frac{\b_\sigma \oplus \b_\sigma}{\g_{1,0} \oplus \g_{0,1}} \right)^* \right)^{-1} \cdot \lambda(\uf_\sigma^*) =  \Omega \left[ \sum_{k = 1}^n A_{\sigma(k)}(q,t) + D_{\sigma(k)}(q,t) \right]^0.
$$
For $1 \le k \le n$, we have
$$
A_{\sigma(k)}(q,t) + D_{\sigma(k)}(q,t) = (1 - q - t + qt) B_{\sigma(k-1)}(q,t) q^{-c_{\sigma}(k)} t^{-r_{\sigma}(k)} - q^{-c_{\sigma}(k)} t^{-r_{\sigma}(k)} - q - t.
$$
Since 
$$
\Omega \left[ - \sum_{k = 1}^n q^{-c_{\sigma}(k)} t^{-r_{\sigma}(k)} \right]^0 = \frac{1}{\Omega[B_\mu(q^{-1},t^{-1})]^0} \quad \textrm{ and } \quad \Omega \left[ - \sum_{k = 1}^n (q + t) \right] = \frac{1}{(1 - q)^{n} (1 - t)^{n}},
$$ 
Theorem \ref{thm:pnpcharactergl} follows.

\subsection{Comparison with the Garsia-Haiman formula}\label{sec:HaimanGarsia}

The Hilbert scheme of $n$ points in the plane, $\Hilb$, is a smooth irreducible variety of dimension $2n$. Haiman has constructed a rank $n!$ vector bundle  $\pp$ on $\Hilb$ called the Procesi bundle. Given a Young diagram $\mu$, let  $\pp_\mu$ be the fiber of $\pp$ at the $\cc$-fixed point $I_\mu \in \Hilb$, a bigraded codimension $n$ ideal $I_\mu \sset \C[x,y]$ associated with $\mu,$ see \cite{HaimanJAMS}.

In the paper \cite{qtwalk}, Garsia and Haiman use an analogue of the Pieri rule for Macdonald polynomials to derive an expression for $\cchi^{\cc}(\pp_{\mu};q,t)$. We recall their result: let $\Rbox_k$, respectively $\Cbox_k$, denote the set of all boxes in the same row, respectively column, of $\sigma(k)$ as the box $p_k = \sigma(k) \backslash \sigma(k-1)$, 
excluding $p_k$ itself. For $x \in Y_{\sigma(k)}$, let $a_k(x)$, respectively $l_k(x)$, denotes the arm length, respectively leg length, of $x$ in $Y_\sigma(k)$. For all $1 \le k \le n$ define
$$
\Pprod_{\sigma(k)}(q,t) := \left( \prod_{x \in \Rbox_k} \frac{(1 - q^{1 + a_k(x)} t^{- l_k(x)})}{(1 - q^{1 + a_{k-1}(x)} t^{- l_{k-1}(x)})} \right) \left( \prod_{x \in \Cbox_k}  \frac{(1 - q^{- a_k(x)} t^{1 + l_k(x)})}{(1 - q^{- a_{k-1}(x)} t^{1 + l_{k-1}(x)})} \right).
$$ 
Then it is shown by Garsia and Haiman \cite[\S $1$]{qtwalk}, that
\beq{eq:GarsiaHaiman}
\cchi^{\cc}(\pp_\mu;q,t) = \sum_{\sigma \in \SYT (\mu)} \, \left( \prod_{k = 1}^n \Pprod_{\sigma(k)}(q,t) \right).
\eeq

\begin{prop}\label{prop:HaimanGarsia}
Let $\sigma$ be a standard tableau of shape $\mu$. Then
\beq{eq:GHcompare}
\frac{\prod_{x \in Y_\mu} (1 - q^{1 + a(x)} t^{- l(x)}) (1 - q^{- a(x)} t^{1 + l(x)})}{(1 - q)^n (1 - t)^n \cdot \Omega[B_\mu(q^{-1},t^{-1})]^0} \cdot \Gamma_{\sigma}(q,t) = \prod_{k = 1}^n \Pprod_{\sigma(k)}(q,t).
\eeq
i.e. the Garsia-Haiman formula is equivalent to equation (\ref{eq:pnpgl}).
\end{prop}

An independent construction of a natural
rank $n!$ vector bundle $\wt \pp$ on  $\Hilb$ is given
in  \cite[\S 8]{Iso}.
It is immediate from the construction of $\wt \pp$
that one has an isomorphism  $\wt\pp_\mu \simeq \rr_\bbe$
where $\wt\pp_\mu$ is the fiber of $\wt\pp$  
at the  point $I_\mu \in \Hilb$ and
 $\bbe$ is the principal nilpotent pair associated
with the diagram $\mu$.
Thus, the Proposition above  says that one has an equality
$\cchi^{\cc}(\pp_\mu;q,t) = \cchi^{\cc}(\wt\pp_\mu;q,t).$
This equality is also a consequence of the $n!$ theorem of Haiman
which
implies, in particular,  a vector bundle isomorphism
$\pp\simeq\wt\pp$, cf.  \cite[\S 5]{HaimanSurvey}.
We note that in order to deduce the
equation $\cchi^{\cc}(\pp_\mu;q,t) =
\cchi^{\cc}(\wt\pp_\mu;q,t)$
one does not actually need the full strength
of the  $n!$ theorem. It suffices to use the recent
result of Gordon \cite{GordonMacdonald}
that insures that the
 Macdonald polynomials can be recovered
from the vector bundle $\wt \pp$ without
the knowledge of the isomorphism $\pp\simeq\wt\pp$.

The remainder of this subsection is devoted to the proof of Proposition
 \ref{prop:HaimanGarsia}, which is just a direct calculation. 

\begin{lem}\label{lem:GH1}
Let $\nu$ be a partition, then
$$
B_\nu(q,t) \cdot (1 - q - t + qt) = 1 + \sum_{i = 1}^{\ell(\nu)} q^{\nu_{i}} (t^{i} -  t^{i-1}) - t^{\ell(\nu)} = 1 + \sum_{j = 1}^{\ell(\nu')} (q^{j} - q^{j-1}) t^{\nu_{j}'} - q^{\ell(\nu')}.
$$
\end{lem}

\begin{proof}
Consider the Young diagram $Y_\nu$ of $\nu$ as a subset of $\Z^2$. The coefficient of $q^u t^v$ in $B_\nu(q,t) \cdot (1 - q - t + qt)$ can be thought of as an integer placed at the point $(u,v) \in \Z^2$. This integer depends on whether the point $(u,v)$ or the point directly below, directly to the left etc. lies in $Y_\nu$ or not. One can check that the only point in $Y_\nu$ whose value is not $0$ is $(0,0)$, whose value is $1$. If we define $\eta$ to be the ``partition'' $(\infty,\nu_1 + 1, \nu_2 + 1, \ds )$ then $Y_\eta \backslash Y_\nu$ is an infinite strip running along the $x$-axis, above the ``diagonal'' edge of $Y_\nu$ and then up the $y$-axis. This strip has two types of corners, those pointing to the south-west like $(0,0)$ in $[2,1]$ and those pointing to the north-east like $(1,1)$ in $[2,2]$. These corners have the value $-1$ and $1$ respectively. All other entries of $Y_\eta \backslash Y_\nu$ (and the rest of $\Z^2$) have value $0$. This picture corresponds to the equation of the lemma.
\end{proof} 

\begin{lem}\label{lem:GH2}
Let $\mu$ be a partition and let $(c,r) \in Y_\mu$ be a removable box. Set $\nu = \mu \backslash \{ (c,r) \}$, then
\beq{eq:expB}
B_\nu(q,t) \cdot (1 - q - t + qt) = 1 - q^c t^r + \sum_{i = 1}^{c} (q^{i} - q^{i-1}) t^{\mu_{i}'} + \sum_{j = 1}^{r} q^{\mu_{j}} (t^{j} - t^{j-1}).
\eeq
\end{lem}

\begin{proof}
Since $(c,r) = (\mu_{r+1},r) = (c, \mu_{c+1}')$, 
\beq{eq:expressionA}
\sum_{i = 1}^{c} (q^{i} - q^{i-1}) t^{\lambda_{i}'} + \sum_{j = 1}^{r} q^{\lambda_{j}} (t^{j} - t^{j-1}) = \sum_{i = 1}^{c} (q^{i } - q^{i-1}) t^{\mu_{i}'} + \sum_{j = 1}^{r} q^{\mu_{j}} (t^{j} - t^{j-1}).
\eeq
Using a pictorial argument as in the proof of Lemma \ref{lem:GH1} one checks that expression (\ref{eq:expressionA}) is the same as the right hand side of (\ref{eq:expB}), except in the boxes $(0,0)$ and $(c,r)$ which have entries $0$. The coefficient of $q^0 t^0$ in the right hand side of (\ref{eq:expB}) is $1$ and the coefficient of $q^c t^r$ is $-1$. Therefore (\ref{eq:expB}) $=$ (\ref{eq:expressionA}) $+ 1 - q^c t^r$. 
\end{proof}

\subsection{}
We write 
$$
S_{\sigma(k)} = \prod_{x \in Y_{\sigma(k)}} (1 - q^{1 + a_k(x)} t^{- l_k(x)}) (1 - q^{- a_k(x)} t^{1 + l_k(x)})
$$
with $S_{\sigma(0)} := 1$, so that $\lambda(\Snw^* \oplus \Sse^*) = S_{\sigma(n)}$. Since
$$
S_{\sigma(n)} = \frac{S_{\sigma(n)}}{S_{\sigma(n-1)}} \cdot \frac{S_{\sigma(n-1)}}{S_{\sigma(n-2)}} \cdots,
$$
we need to show that 
$$
\Pprod_{\sigma(k)}(q,t) = \frac{S_{\sigma(k)}}{S_{\sigma(k-1)}} \cdot  \Omega \left[ -q - t - q^{-c_{\sigma}(k)} t^{-r_{\sigma}(k)} + (1 - q - t + qt) B_{\sigma(k-1)}(q,t) q^{-c_{\sigma}(k)} t^{-r_{\sigma}(k)} \right]^0,
$$
for all $1 \le k \le n$. The argument is similar for all $k$ so we take $k = n$. 

\subsection{} 
Since 
$$
\frac{S_{\sigma(n)}}{S_{\sigma(n-1)}} = \left( \prod_{x \in \Rbox_n \cup \Cbox_n} \frac{(1 - q^{1 + a_n(x)} t^{- l_n(x)})(1 - q^{- a_n(x)} t^{l_n(x) + 1})}{(1 - q^{1 + a_{n-1}(x)} t^{- l_{n-1}(x)})(1 - q^{- a_{n-1}(x)} t^{1 + l_{n-1}(x)})} \right) \cdot (1 - q)(1 - t), 
$$
we must show that $\Omega \left[ - q^{-c_{\sigma}(k)} t^{-r_{\sigma}(k)} + (1 - q - t + qt) B_{\sigma(k-1)}(q,t) q^{-c_{\sigma}(k)} t^{-r_{\sigma}(k)} \right]^0 =$
$$
\left( \prod_{x \in \Cbox_n} \frac{(1 - q^{1 + a_{n-1}(x)} t^{- l_{n-1}(x)})}{(1 - q^{1 + a_{n}(x)} t^{- l_{n}(x)})} \right) \left( \prod_{x \in \Rbox_n}  \frac{(1 - q^{- a_{n-1}(x)} t^{1 + l_{n-1}(x)})}{(1 - q^{- a_{n}(x)} t^{1 + l_{n}(x)})} \right).
$$
The right hand side of the above equation is $\Omega$ applied to 
\beq{eq:answer}
\sum_{x \in \Cbox_n} q^{1 + a_{n-1}(x)} t^{- l_{n-1}(x)} - q^{1 + a_{n}(x)} t^{- l_{n}(x)} + \sum_{x \in \Rbox_n}  q^{- a_{n-1}(x)} t^{1 + l_{n-1}(x)} - q^{- a_{n}(x)} t^{1 + l_{n}(x)}
\eeq

\begin{lem}\label{lem:GH3}
The expression (\ref{eq:answer}) equals
\beq{eq:answer2}
q^{-c} t^{-r} \left( \sum_{i = 1}^{c} (q^{i } - q^{i-1}) t^{\lambda_{i}'} + \sum_{j = 1}^{r} q^{\lambda_{j}} (t^{j} - t^{j-1}) \right),
\eeq
where $c := c(n)$ and $r := r(n)$.
\end{lem}

\begin{proof}
Note that $\Rbox_n = \{ (i,r) \, | \, 0 \le i < c \}$ and $\Cbox_n = \{ (c,j) \, | \, 0 \le j < r \}$. Then, for $i < c$ and $j < r$,
$$
a_n(i,r) = c - i, \quad a_{n-1}(i,r) = c - i - 1, \quad a_n(c,j) = a_{n-1}(c,j) = \lambda_{j+1} - c - 1,
$$
$$
l_n(i,r) = l_{n-1}(i,r) = \lambda_{i+1}' - r - 1, \quad l_n(c,j) = r - j, \quad l_{n-1}(c,j) = r - j - 1.
$$
Equation (\ref{eq:answer2}) follows.
\end{proof}

The statement of Proposition \ref{prop:HaimanGarsia} follows by combining Lemmata \ref{lem:GH2} and \ref{lem:GH3}.

\section{Polygraph spaces}\label{sec:polygraph}

\subsection{Two parameter Macdonald polynomials and $(q,t)$-Kostka polynomials}

The combinatorics of the Procesi bundle on the Hilbert scheme is described by transformed Macdonald polynomials. We denote by $\tilde{H}_\mu [Z ; q,t]$ the \textit{transformed Macdonald polynomial} labeled by the partition $\mu$, as defined in \cite[Definition 3.5.2]{HaimanSurvey}. The transformed Macdonald polynomials are related to the integral form $J_\mu(Z;q,t)$ of Macdonald's polynomials as follows:
\beq{eq:integralMacdonald}
J_\mu(Z;q,t) = t^{n(\mu)} \tilde{H}_\mu [ ( 1- t^{-1}) Z ; q,t^{-1}].
\eeq

Since the polynomials $\tilde{H}_\mu [Z ; q,t]$ and $J_{\mu}(Z;q,t)$ are symmetric, there exist polynomials $\tilde{K}_{\lambda,\mu}(q,t)$ and $K_{\lambda,\mu}(q,t)$ such that 
\beq{eq:KostkaMacdonaldpolynomials}
\tilde{H}_\mu [ Z ; q,t] = \sum_{\lambda \vdash n} \tilde{K}_{\lambda,\mu}(q,t) \cdot s_\lambda(z), \qquad J_\mu ( Z ; q,t) = \sum_{\lambda \vdash n} K_{\lambda,\mu}(q,t) \cdot s_\lambda \left[ (1 - t)Z \right].
\eeq
The polynomials $\tilde{K}(q,t)$ are called the \textit{Kostka-Macdonald polynomials}. Equation (\ref{eq:integralMacdonald}) implies that $\tilde{K}(q,t) = t^{n(\mu)} K_{\lambda,\mu}(q,t^{-1})$. A direct corollary of Haiman's proof of the $n !$ conjecture, and the original motivation for the work, is the confirmation of Macdonald's positivity conjecture: $\tilde{K}_{\lambda,\mu}(q,t) \in \Nat[q,t], \ \forall \ \lambda, \mu \vdash n$.

\subsection{}\label{sec:hilbert}
Recall (\ref{sec:HaimanGarsia}) that $\Hilb$ denotes the Hilbert scheme of $n$ points in the plane. Let  $\taut$ denote the tautological bundle on $\Hilb$, whose fiber over a point $I \in \Hilb$ is $\C[x,y] / I$. Since the action of $\s_m$ on $\Hilb$ is trivial, the vector bundle $\taut^{\o m}$ decomposes as 
$$
\taut^{\o m} = \bigoplus_{\mu \vdash m; \ \ell(\mu) \le n.} \taut_\mu \o \chi_\mu.
$$
Write $R(n,\mu) = H^0(\Hilb,\pp \o \taut_\mu)$. By \cite[Theorem
3.5]{HaimanVanishing}, the bigraded $W$-character of $R(n,m)$ is given
by, cf. \eqref{bnu}
\beq{eq:Haiman1}
\cchi^{W \times \cc}(R(n,m)) = \sum_{\nu \vdash n} \frac{B_\nu(q,t)^m \cdot \tilde{H}_{\nu}[Z;q,t]}{\prod_{x \in Y_{\nu}}(1 - t^{1 + l(x)} q^{- a(x)}) (1 - t^{- l(x)} q^{1 + a(x)})}.
\eeq

Choose some ordering $\{(r_1,s_1), \ds , (r_n,s_n) \}$ of the Young diagram $Y_{\nu}$. If $s_\mu(z_1, \ds , z_n)$ is a Schur polynomial then the plethystic substitution $s_\mu \left[ B_\nu(q,t) \right]$ is the polynomial in $q,t$ obtained by evaluating $s_\mu$ at $z_1 = q^{r_1} t^{s_1}, \ds, z_n = q^{r_n} t^{s_n}$. Since the fiber of the vector bundle $\taut_\mu$ at the fixed point of $\Hilb$ labeled by $\nu$ has bigraded character $s_\mu [B_{\nu}(q,t)]$, formula (\ref{eq:Haiman1}) implies that 
\beq{eq:Haiman2}
\cchi^{W \times \cc}(R(n,\mu)) = \sum_{\nu \vdash n} \frac{s_\mu \left[ B_\nu(q,t) \right] \cdot \tilde{H}_{\nu}[Z;q,t]}{\prod_{x \in Y_{\nu}}(1 - t^{1 + l(x)} q^{- a(x)}) (1 - t^{- l(x)} q^{1 + a(x)})}.
\eeq

On the other hand, according to  \cite[Theorem 1.9.1]{Iso}, for each $m \ge 0$ there is a $\cc \times W \times \s_m$-equivariant isomorphism of $\C[\tt]$-modules
\beq{eq:maingeoequation}
R(n,m) \simeq(\rr \o \C^m[V] \o V^{\o m})^{SL_n(\C)},
\eeq
where $\s_m$ acts by permuting the tensorands of $\taut^{\o m}$ on the left-hand side and permutes the tensorands of $V^{\o m}$ on the right-hand side.

Since each $G \times W \times \cc$-isotypic component of $\rr$ is finite dimensional, we can write
$$
\cchi^{G \times W \times \cc}(\rr) = \sum_{\mu \in P^+_n; \, \lambda
\vdash n.} \,
 b_{\mu,\lambda}(q,t) \cdot s_{\mu}(z) \cdot \chi_{\lambda}.
$$

Comparison of formulas \eqref{eq:Haiman2} and \eqref{eq:maingeoequation} allows us to obtain an expression for some of the polynomials $ b_{\mu,\lambda}(q,t)$ in terms of Macdonald polynomials. This is the content of Theorem \ref{thm:Pone} below.

\subsection{} 
Let $\Pnone$ be the set of all partitions with at most $n-1$ parts, thought of as those weights $\mu \in P^+_n$ such that $\mu_n = 0$. Define a partial ordering on $\Pnone$ by setting $\mu \preceq \lambda$ if $Y_\mu \subset Y_\lambda$. We denote by $\psi$ the $\Z$-linear operator on $\Z[\Pnone]$ defined by $\psi(\lambda) = \sum_{\mu} \mu$, where the sum is over all $\mu \preceq \lambda$ such that $\lambda / \mu$ is a skew partition not containing the vertical strip $(1,1)$ (so $\lambda / \mu$ consists entirely of disjoint horizontal strips). The operator $\psi$ is invertible. Its inverse will be denoted $\phi$. 

\begin{examp}  
Let $\lambda = (2,2,1,0)$. Then 
$$
\psi(\lambda) = (2,2,1,0) + (2,2,0,0) + (2,1,1,0) + (2,1,0,0)
$$ 
and 
$$
\phi(\lambda) = (2,2,1,0) - (2,2,0,0) - (2,1,1,0) + (2,1,0,0) + (1,1,1,0) - (1,1,0,0).
$$
\end{examp}

The map $\lambda \mapsto s_\lambda(z)$ realizes $\Z[\Pnone]$ as a
$\Z$-submodule of $\Lambda$. The image of $\phi(\lambda)$, respectively
$\psi(\lambda)$, under this map will be denoted $\Phi_\lambda(z)$,
respectively $\Psi_{\lambda}(z)$. Denote by $w_0$ the longest word in
$\s_n$ so that $(V_\lambda)^* \simeq V_{- w_0(\lambda)}$ for all
$\lambda$ in $P^+_n$. We denote by $\star$ the map from $\Pnone$ to
$Q^+_n$ that sends $\mu$ to 

$$
\mu^\star := -w_0(\mu_1, \ds,\mu_{n-1}, - \sum_{i < n} \mu_i).
$$
This map is injective but clearly not bijective.   

The main result of this section is the following, whose proof is given in the next section.

\begin{thm}\label{thm:Pone}
Let $\mu \in \Pnone$ and $\lambda \vdash n$, then
\beq{eq:mainformula}
b_{\mu^\star,\lambda}(q,t) = \sum_{\nu \vdash n} \frac{\Phi_\mu \left[ B_\nu(q,t) \right] \cdot  \tilde{K}_{\lambda,\nu}(q,t)}{\prod_{x \in Y_{\nu}}(1 - t^{1 + l(x)} q^{- a(x)}) (1 - t^{- l(x)} q^{1 + a(x)})}.
\eeq
\end{thm} 

\begin{cor}\label{cor:zznormchar}
Let $\mu \in \Pnone$, then the bigraded character of the $\mu^\star$-isotypic component of $\mathcal{O}_{\zz_{\norm}}$ is given by 
$$
\cchi^{\cc}((\mathcal{O}_{\zz_{\norm}})_{\mu^{\star}}) = \sum_{\nu \vdash n} \frac{\Phi_\mu \left[ B_\nu(q,t) \right]}{\prod_{x \in Y_{\nu}}(1 - t^{1 + l(x)} q^{- a(x)}) (1 - t^{- l(x)} q^{1 + a(x)})}.
$$
\end{cor}

\begin{proof}
By \cite[Corollary 1.5.1 (i)]{Iso}, we have $\mathcal{O}_{\zz_{\norm}} \simeq \rr^W$. Therefore $\cchi^{\cc}((\mathcal{O}_{\zz_{\norm}})_{\mu^{\star}})$ is given by (\ref{eq:mainformula}) with $\lambda = (n)$. But, as noted in \cite[\S 3.5]{HaimanSurvey}, $\tilde{K}_{(n),\nu}(q,t) = 1$ for all $\nu$. 
\end{proof} 

Since $\Phi_{(0)}(z) = s_{(0)}(z) = 1$ and $(0)^\star = (0)$, Theorem \ref{thm:Pone} also implies:

\begin{cor}\label{cor:anotherproof}
The bigraded $\s_n$-character of $\rr^G$ is
$$
\cchi^{\s_n \times \cc}(\rr^G) = \sum_{\mu \vdash n} \frac{\tilde{H}_{\mu}[Z,q,t]}{\prod_{x \in Y_{\mu}}(1 - t^{1 + l(x)} q^{- a(x)}) (1 - t^{- l(x)} q^{1 + a(x)})}.
$$
\end{cor}

\subsection{} The proof of Theorem \ref{thm:Pone} is based on a version of the Pieri rules. In order to agree with the convention in (\ref{sec:partitions}), we think of $\mu \in P_n^+$ as $n$ infinite rows, bounded on the right, such that the lengths of the rows decrease as we go up the page. For $m \in \Nat$ and $\mu \in P^+_n$, define $E(\mu,m)$ to be the set of all $\nu \in P^+_n$ obtained from $\mu$ by removing $m$ squares, no two from the same column. For instance, if $\mu = (4,2,-1)$ and $m = 3$ then 
$$
E(\mu,m) = \{ (4,2,-4),(4,1,-3),(4,0,-2),(3,2,-3),(3,1,-2),(3,0,-1),(2,2,-2),(2,1,-1) \}.
$$

\begin{lem}\label{lem:inversePieri}
Let $\mu \in P^+_n$ and $m \in \Nat$, then
$$
\C^m[V] \o V_\mu \simeq \bigoplus_{\lambda \in E(\mu,m)} V_\lambda.
$$
\end{lem}

\begin{proof}
Define $D(\mu,m)$ to be the set of all dominant weights $\lambda$ obtained from the weight $\mu$ by adding $m$ boxes, no two in the same column. Then Pieri's rule says that
$$
\sym^m V \o V_\mu \simeq \bigoplus_{\lambda \in D(\mu,m)} V_\lambda.
$$
Choose $k >\!> 0$ and let $\eta = (k^n) - w_0(\mu)$, then
$$
\C^m[V] \o V_\mu = (\sym^m V \o V_{-w_0(\mu)})^* = (\sym^m V \o V_{-w_0(\mu)} \o \mathrm{det}^k \o \mathrm{det}^{-k})^*
$$
$$
\simeq \left( \bigoplus_{\lambda \in D(\eta,m)} V_\lambda \o \mathrm{det}^{-k} \right)^* = \left( \bigoplus_{\lambda \in D(\eta,m)} V_{\lambda -(k^n)} \right)^* = \bigoplus_{\lambda \in D(\eta,m)} V_{-w_0(\lambda -(k^n))}.
$$
This is the same as the statement of the lemma.
\end{proof}

\begin{prop}\label{prop:equalityone}
Let $\mu \in \Pnone$, $\mu \vdash m$, then
\beq{eq:Esumcharacter}
\sum_{\stackrel{-w_0(\eta) \in E(\mu,m)}{\lambda \vdash n}} b_{\eta,\lambda}(q,t) \cdot \chi_{\lambda} = \sum_{\nu \vdash n} \frac{s_\mu \left[ B_\nu(q,t) \right] \cdot \tilde{H}_{\nu}[Z;q,t]}{\prod_{x \in Y_{\nu}}(1 - t^{1 + l(x)} q^{- a(x)}) (1 - t^{- l(x)} q^{1 + a(x)})}.
\eeq
\end{prop}

\begin{proof}
Equating the multiplicity spaces of the irreducible $\s_m$-module $\chi_\mu$ in (\ref{eq:maingeoequation}) gives an isomorphism of $\cc \times W$-modules
\beq{eq:randomname}
(\rr \o \C^m[V] \o V_{\mu})^{SL_n(\C)} \simeq R(n,\mu).
\eeq
By Lemma \ref{lem:inversePieri}, $\C^m[V] \o V_{\mu}$ decomposes into a direct sum of those irreducible $GL_n(\C)$-modules $V_\eta$ such that $\eta \in E(\mu,m)$. Let $\eta, \nu  \in P^+_n$, then $(V_\eta \o V_\nu)^{SL_n(\C)} \neq 0$ if and only if there exists $k \in \Z$ such that $(k^n) = \eta + w_0(\nu)$ i.e. $\eta_i + \nu_{n - i} = k$ for all $i \in \{ 1 , \ds , n \}$. In this case the space $(V_\eta \o V_\nu)^{SL_n(\C)}$ is one-dimensional. If both $\eta, \nu \in Q^+_n$ then $(V_\eta \o V_\nu)^{SL_n(\C)} \neq 0$ if and only if $\nu = - w_0(\eta)$. From the definition of $\rr$ we see that $b_{\nu,\lambda}(q,t) = 0$ unless $\nu \in Q^+_n$. Similarly, if $\eta \in E(\mu,m)$ then $\eta \in Q^+_n$. Therefore the stated formula follows from equations (\ref{eq:randomname}) and (\ref{eq:Haiman2}) and Lemma \ref{lem:inversePieri}.
\end{proof}

\begin{proof}[Proof of Theorem \ref{thm:Pone}]
Firstly, note that taking the $\chi_\lambda$-isotypic component of (\ref{eq:Esumcharacter}) and using the expansion (\ref{eq:KostkaMacdonaldpolynomials}) of $\tilde{H}_{\nu}(z;q,t)$ in terms of the Kostka-Macdonald polynomials $\tilde{K}_{\lambda,\nu}(q,t)$ gives
\beq{eq:eq12}
\sum_{-w_0(\eta) \in E(\mu,|\mu|)} b_{\eta,\lambda}(q,t) = \sum_{\nu \vdash n} \frac{s_\mu \left[ B_\nu(q,t) \right] \cdot \tilde{K}_{\lambda,\nu}(q,t)}{\prod_{x \in Y_{\nu}}(1 - t^{1 + l(x)} q^{- a(x)}) (1 - t^{- l(x)} q^{1 + a(x)})}.
\eeq
By definition, $\{ -w_0(\eta) \ | \ \eta \in E(\mu,|\mu|) \} = \{ \varphi^\star \ | \ \varphi \preceq \mu \}$. Therefore if we take the $\Z$-linear sum of equation (\ref{eq:eq12}) over all terms in $\phi(\mu)$ the left hand side becomes $b_{\mu^\star,\lambda}(q,t)$ and the right hand side is 
$$
\sum_{\nu \vdash n} \frac{\Phi_\mu \left[ B_\nu(q,t) \right] \cdot  \tilde{K}_{\lambda,\nu}(q,t)}{\prod_{x \in Y_{\nu}}(1 - t^{1 + l(x)} q^{- a(x)}) (1 - t^{- l(x)} q^{1 + a(x)})}
$$ 
as required.
\end{proof}

\subsection{The sheaf $\rrnil$}\label{rnilsec}
As explained in \cite[\S 9]{Iso}, we have a commutative diagram
$$
\xymatrix{
\xx_{\norm} \ar[r]^<>(0.5){\Oflat} \ar[d]_{p_{\norm}} & \t \ar[d]^{\pi} \\
\zz_{\norm} \ar[r]^<>(0.5){\Ozflat} & \t / W
}
$$
where $\Oflat : \xx_{\norm} \rightarrow \t$ is the morphism induced by
the projection from $\gg \times \tt$ onto the first copy of $\t$ and
$\Ozflat$ is the map induced from projection of $\gg$ onto the first
copy of $\g$, followed by the quotient map $\g \rightarrow \t / W$. It
is shown in \cite[Proposition 9.1.3]{Iso} that the maps $\Oflat$ and
$\Ozflat$ are flat. Define $\xxnil$ and $\zznil$ to be the
scheme-theoretic fibers $\Oflat^{-1}(0)$ and $\Ozflat^{-1}(0)$
respectively. 

We write $\rrnil$ for the coherent sheaf $(p_{\norm})_*
\mathcal{O}_{\xxnil}$ on $\zznil$. The group $H \times W$ acts on  $\xxnil$.
For each $\mu \in \Irr (W)$, let $\rrnil_{\mu}$ be the $\chi_\mu$-isotypic component
 of $\rrnil$ so that $\rrnil = \bigoplus_{\mu \in \Irr (W)}
\rrnil_{\mu} \o \chi_\mu$, where the action of $W$ on $\rrnil_{\mu}$ is
trivial.\\

For $\lambda, \mu \vdash n$, define polynomials ${\mathbf k}_{\lambda,\mu}(q,t) \in \Q[q,t]$ by
$$
\tilde{H}_\mu [(1 - q) Z;q,t] = \sum_{\lambda \vdash n} {\mathbf k}_{\lambda,\mu} (q,t) \cdot s_{\lambda} (z).
$$

\begin{cor}
Let $\mu \in \Pnone$ and $\lambda \vdash n$, then 
\beq{eq:Phinil}
\cchi^{\cc}((\rrnil_\lambda)_{\mu^\star}) = \sum_{\nu \le \lambda} \frac{\Phi_\mu \left[ B_\nu(q,t) \right] \cdot {\mathbf k}_{\lambda,\nu}(q,t)}{\prod_{x \in Y_{\mu}}(1 - t^{1 + l(x)} q^{- a(x)}) (1 - t^{- l(x)} q^{1 + a(x)})}
\eeq
and specializing to $t = q^{-1}$:
\beq{eq:specialisePhi}
\cchi^{\Cs}((\rrnil_\lambda)_{\mu^\star}) = q^{-n(\lambda)} \Phi_\mu \left[ B_\lambda(q,q^{-1}) \right] H_\lambda(q^{-1})^{-1},
\eeq
where $(\rrnil_\lambda)_{\mu^\star} = (\rrnil_\lambda \o V_{-w_0(\mu^\star)})^G$ is the multiplicity space of $V_{\mu^\star}$ in $\rrnil_\lambda$.
\end{cor}

\begin{proof}
Let $\t$ denote the permutation representation of $\s_n$, equipped with an action of $\cc$ such that $\t$ has bi-degree $(1,0)$. Let $M$ be a $\s_n \times \cc$-equivariant $\C[\t]$-module. If $M$ is flat as a $\t$-module, then using the Koszul resolution of $M$ it is shown in \cite[Proposition 3.3.1]{HaimanSurvey} that 
$$
\cchi^{\s_n \times \cc}(M / \t \cdot M; Z) = \cchi^{\s_n \times \cc}(M ; (1 - q) Z).
$$
A key result in Haiman's work on the isospectral Hilbert scheme is \cite[Proposition 3.8.1]{HaimanJAMS} which implies that $R(n,m)$ is a flat $\C[\t]$-module. Therefore the direct summand $R(n,\mu)$ of $R(n,m)$ is also flat over $\t$ and hence formulae (\ref{eq:Haiman2}) implies that  
\beq{eq:Koszul1}
\cchi^{\s_n \times \cc}(R(n,\mu) / \t \cdot R(n,\mu)) = \sum_{\nu \vdash n} \frac{s_\mu \left[ B_\nu(q,t) \right] \cdot \tilde{H}_{\nu}[(1-q)Z;q,t]}{\prod_{x \in Y_{\nu}}(1 - t^{1 + l(x)} q^{- a(x)}) (1 - t^{- l(x)} q^{1 + a(x)})}.
\eeq
Note that by \cite[Definition 3.5.2 (i)]{HaimanSurvey}, ${\mathbf k}_{\lambda,\nu} (q,t) = 0$ unless $\nu \le \lambda$. Then equation (\ref{eq:Koszul1}) implies that the $\chi_\lambda$-isotypic component of $R(n,\mu) / \t \cdot R(n,\mu)$ has bigraded character 
$$
\sum_{\nu \le \lambda} \frac{s_\mu \left[ B_\nu(q,t) \right] \cdot {\mathbf k}_{\lambda,\nu}(q,t)}{\prod_{x \in Y_{\nu}}(1 - t^{1 + l(x)} q^{- a(x)}) (1 - t^{- l(x)} q^{1 + a(x)})}.
$$
Now if we replace the right hand side of equation (\ref{eq:eq12}) with the above expression and repeat the argument given in the proof of Theorem \ref{thm:Pone} we get expression (\ref{eq:Phinil}). 
Using the identity, \cite[Proposition 3.5.10]{HaimanSurvey}, 
$$
\tilde{H}_\lambda[Z;q,q^{-1}] = q^{-n(\lambda)} H_\lambda(q) \cdot s_\lambda \left[ \frac{Z}{(1 - q)} \right]
$$
and making the specialization $t = q^{-1}$ in equation (\ref{eq:Koszul1}) gives 
\beq{eq:Koszul2}
\cchi^{\s_n \times \Cs}(R(n,\mu) / \t \cdot R(n,\mu)) = \sum_{\lambda \vdash n} q^{-n(\lambda)} s_\mu \left[ B_\lambda(q,q^{-1}) \right] H_\lambda(q^{-1})^{-1} \cdot \chi_\lambda.
\eeq
Here we have used the fact that 
$$
\left[ \prod_{x \in Y_{\lambda}}(1 - t^{1 + l(x)} q^{- a(x)}) (1 - t^{- l(x)} q^{1 + a(x)}) \right] \bigm\arrowvert_{t = q^{-1}} = H_{\lambda}(q) \cdot H_{\lambda}(q^{-1}).
$$
Once again, we replace the right hand side of equation (\ref{eq:eq12}) with equation (\ref{eq:Koszul2}) and repeat the argument given in the proof of Theorem \ref{thm:Pone} to get (\ref{eq:specialisePhi}).
\end{proof}

\section{Rational Cherednik algebras}\label{sec:rationalCherednik}

In this section we use results of Rouquier, Varagnolo-Vasserot and Leclerc-Thibon to describe a character formula for certain simple modules of the rational Cherednik algebra of type $\mathbf A$ that belong to category $\mathcal{O}$. Via the Calaque-Enriquez-Etingof functor, see (\ref{sec:CEEfunctor}), this allows us to give a character formula for those $GL_n(\C)$-equivariant $\dd(\mathfrak{gl}_n)$-modules that are supported on the nilpotent cone. 

\subsection{The rational Cherednik algebra associated to the symmetric group}
Let $k \in \Q$. The rational Cherednik algebra $\Ch{m}{k}$ associated to the symmetric group $\s_m$ is defined to be the quotient of $\C \langle \mbf{x} \rangle \o \C \langle \mbf{y} \rangle \rtimes \s_m$, where $\C \langle \mbf{x} \rangle := \C \langle x_1, \ds, x_m \rangle$ and $\C \langle \mbf{y} \rangle := \C \langle y_1, \ds, y_m \rangle$, by the relations
$$
[x_i,x_j] = 0, \quad [y_i,y_j] = 0, \quad \forall \, i,j,
$$
and 
$$
[y_i, x_j ] = k s_{ij}, \quad \forall \, i \neq j \quad \ [y_i,x_i] = 1 - k \sum_{j \neq i} s_{ij}. 
$$
By \cite[Theorem 1.3]{EG}, the rational Cherednik algebra has a triangular decomposition $\Ch{m}{k} \simeq \C[ \mbf{x}] \o \C [\s_m] \o \C [ \mbf{y}]$, where $\C[ \mbf{x}] := \C [ x_1, \ds, x_m ]$ and $\C [ \mbf{y} ] := \C [ y_1, \ds, y_m ]$. Category $\catO{m}{k}$ is define to be the category of all $\Ch{m}{k}$-modules $M$ such that 
\begin{enumerate}
\item $M$ is finitely generated as a $\C[\mbf{x}]$-module,
\item the action of $\C[\mbf{y}]$ on $M$ is locally nilpotent. 
\end{enumerate}
This is a highest weight category and the standard modules of this category are
$$
\Delta(\lambda) := \Ch{m}{k} \o_{\C[\mbf{y}] \rtimes \s_m} \lambda,
$$
where $\lambda$ is an irreducible $\s_m$-module, extended to a $\C[\mbf{y}] \rtimes \s_m$-module by making the $y_i$ act as zero. The simple head of $\Delta(\lambda)$ will be denoted $L(\lambda)$ so that the set $\{ L(\lambda) \, | \, \lambda \in \mathrm{Irr} (\s_m) \}$ is a complete set of non-isomorphic simple modules in $\catO{m}{k}$.

\subsection{Character of simple $\Ch{m}{k}$-modules}
Let
$$
\adh = \frac{1}{2} \sum_{i = 1}^m x_i y_i + y_i x_i = \sum_{i = 1}^m x_i y_i + \frac{m}{2} - \frac{k}{2} \sum_{1 \le i \neq j \le m} s_{ij}.
$$
Then $[\adh,x_i] = x_i$ and $[\adh,y_i] = -y_i$. The element $\adh$ defines a grading on those $\Ch{m}{k}$-modules $M$ such that the action of $\adh$ is locally finite with each generalized eigenspace
$$
M_{\alpha} := \{ m \in M \, | \, (\adh - \alpha)^l \cdot m = 0, \, l >\!> 0 \}
$$
being finite dimensional. If $M \in \catO{m}{k}$ then $\dim M_\alpha < \infty$ for all $\alpha$ and one can define the character of $M$ to be
$$
\cchi^{\adh}(M) = \sum_{\alpha \in \C} t^{\alpha} \dim M_\alpha \in \bigoplus_{\alpha \in \C/ \mathbb{Z}} t^{\alpha} \cdot \Nat((t)).
$$ 
We denote by $\kappa(k,\lambda)$ the scalar by which $\adh$ acts on the lowest weight vectors in $\Delta(\lambda)$. Then $\cchi^{\adh}(\Delta(\lambda))$ and $\cchi^{\adh} (L(\lambda))$ belong to $t^{\kappa(k,\lambda)} \Nat[[t]]$. 

\begin{lem}\label{lem:kappacalculation}
For all $k \in \Q$,
\beq{eq:kappavalue}
\kappa(k,\lambda) = \frac{m}{2} + k ( n(\lambda) - n(\lambda')).
\eeq
\end{lem}

\begin{proof}
Recall that the \textit{Jucys-Murphy} elements in $\C [ \s_m]$ are defined to be $\Theta_i = \sum_{j < i} s_{ij}$, for all $i = 2, \ds, m$ so that
$$
\adh = \sum_{i = 1}^m x_i y_i + \frac{m}{2} - k \sum_{i = 2}^m \Theta_i.
$$
Let $\sigma$ be a standard tableau of shape $\lambda$ and $v_\sigma$ the corresponding vector in $\chi_\lambda$. Then 
$$
\Theta_i \cdot v_\sigma = \mathrm{ct}_\sigma(i) v_\sigma,
$$
where $c_\sigma(i)$ is the column of $\lambda$ containing $i$, $r_\sigma(i)$ is the row of $\lambda$ containing $i$ and $\mathrm{ct}_\sigma(i) := c_{\sigma}(i) - r_\sigma(i)$ is the content of the node containing $i$. Note that $\mathrm{ct}_\sigma(1)= 0$ for all standard tableaux $\sigma$. Therefore 
$$
\adh \cdot v_\sigma = \left(  \frac{m}{2} - k \sum_{i = 2}^m \mathrm{ct}_\sigma(i) \right) v_\sigma,
$$
and hence
$$
\kappa(k,\lambda) = \frac{m}{2} - k \sum_{i = 2}^m \mathrm{ct}_\sigma(i) = \frac{m}{2} - k \sum_{i = 1}^n \mathrm{ct}_\sigma(i).
$$
Now $\sum_{i = 1}^m r_\sigma(i) = \sum_{j = 1}^{\ell(\lambda)} (j - 1) \lambda_j = n(\lambda)$ and similarly $\sum_{i = 1}^m c_\sigma(i) = n(\lambda')$. This implies equation (\ref{eq:kappavalue}).
\end{proof}

\subsection{Rouquier's equivalence} Let $k \in \mathbb{Q}_{> 0}$ and fix $\eta = \mathrm{exp}(2 \pi \sqrt{-1} k)$ to be a primitive $r^{th}$ root of unity. Let $\Schur_{\eta}(m)$ be the quantized Schur algebra of type $\mathbf A$, specialized to $\eta$ (see \cite[\S 6.2]{RouquierQSchur}). It is a finite dimensional $\C$-algebra and the category of finitely generated left $\Schur_{\eta}(m)$-modules is a highest weight category. The standard and simple modules in this highest weight category are naturally labeled by partitions of $m$. Denote the standard, respectively simple, module labeled by $\lambda$ by $W_\lambda$, respectively $L_\lambda$. 

It was conjectured by Leclerc and Thibon \cite{LT}, and proved by Varagnolo and Vasserot \cite{VaragnoloVasserotDecomposition}, that the decomposition matrix for this highest weight category can be expressed in terms of the transition matrix between the standard and canonical basis of the level one Fock space. Let $\mathcal{F}_{\upsilon}$ be the level one Fock space for the quantum affine algebra $\mathcal{U}_{\upsilon}(\widehat{\sl}_r)$. It is a $\mathbb{Q}(\upsilon)$ vector space with standard basis $\{ | \lambda \rangle \}$, labeled by all partitions. Let $\mathcal{L}^+$ (respectively $\mathcal{L}^-$) be the $\Z[\upsilon]$-sublattice (respectively $\Z[\upsilon^{-1}]$-sublattice) spanned by $\{ | \lambda \rangle \}$. Leclerc and Thibon (see \cite[Theorem 4.1]{LT}) constructed canonical basis $\{ \mathcal{G}^+(\lambda) \}$ and $\{ \mathcal{G}^-(\lambda) \}$ such that $\mathcal{G}^+(\lambda) \equiv | \lambda \rangle \mod \upsilon \Z[\upsilon]$ and $\mathcal{G}^-(\lambda) \equiv | \lambda \rangle \mod \upsilon^{-1} \Z[\upsilon^{-1}]$. Set
$$
\mathcal{G}^+(\mu) = \sum_{\lambda} d_{\lambda,\mu}(\upsilon) |\lambda \rangle, \qquad \mathcal{G}^-(\lambda) = \sum_\mu e_{\lambda,\mu}(\upsilon) | \mu \rangle.
$$
The polynomials $d_{\lambda,\mu}$ and $e_{\lambda,\mu}$ have the following properties: they are non-zero only if $\lambda$ and $\mu$ have the same $r$-core, $d_{\lambda,\lambda}(\upsilon) = e_{\lambda,\lambda}(\upsilon) = 1$, and $d_{\lambda,\mu}(\upsilon) = 0$ unless $\lambda \le \mu$, and $e_{\lambda,\mu}(\upsilon) = 0$ unless $\mu \le \lambda$. Then, assuming that $r > 1$, \cite[Theorem 11]{VaragnoloVasserotDecomposition} says that 
$$
[W_\lambda : L_\mu] = d_{\lambda',\mu'}(1), \qquad [L_\lambda : W_\mu] = e_{\lambda,\mu}(1).
$$

Rouquier has shown in \cite[Theorem 6.11]{RouquierQSchur} that there is an equivalence of highest weight categories $\catO{m}{k} \stackrel{\sim}{\longrightarrow} \Lmod{\Schur_\eta(m)}$ provided $k \notin \frac{1}{2} + \Z$. Under the  condition $k \ge 0$, this equivalence\footnote{Note that Rouquier's rational Cherednik algebra is parameterized by $h = - k$.} sends $\Delta(\lambda)$ to $W_{\lambda}$ and $L(\lambda)$ to $L_{\lambda}$. It is noted in \cite[Remark 6.9]{RouquierQSchur} that the restriction $k \notin \frac{1}{2} + \Z$ is probably un-necessary. Therefore, to make our presentation clearer, the following assumption is made:

\begin{assume}
For all $k \ge 0$, there is an equivalence of highest weight categories\\ $\catO{m}{k} \stackrel{\sim}{\rightarrow} \Lmod{\Schur_\eta(m)}$, sending $\Delta(\lambda)$ to $M_\lambda$ and $L(\lambda)$ to $L_\lambda$.
\end{assume}  

Thus,
\beq{eq:decompmatrix}
[\Delta(\lambda) : L(\mu)] = d_{\lambda',\mu'}(1), \quad \textrm{ and } \quad [L(\lambda) : \Delta(\mu)] = e_{\lambda,\mu}(1).
\eeq

\subsection{}
For the remainder of section \ref{sec:rationalCherednik} we break with convention and represent the character of the irreducible $\s_m$-module labeled by $\lambda$ by $s_\lambda(z)$. The reasons for this are, firstly, that we wish to use plethystic substitutions and, secondly, in section \ref{sec:CEEfunctor} we will apply the Calaque-Enriquez-Etingof functor, which is based on Schur-Weyl duality. This way, the formulae become more manageable. Using plethystic substitutions and the vector space isomorphism $\Delta(\lambda) \simeq \C[\mbf{x}] \o \lambda$, 
 one can write
$$
\cchi^{\adh \times \s_m}(\Delta(\lambda)) = t^{\kappa(k,\lambda)} s_\lambda \left[ \frac{Z}{(1 - t)} \right].
$$
Therefore
\beq{eq:simplechar1}
\cchi^{\adh \times \s_m}(L(\lambda)) = \sum_{\mu \vdash m} e_{\lambda,\mu}(1) t^{\kappa(k,\mu)} s_\mu \left[ \frac{Z}{(1 - t)} \right]. 
\eeq

\begin{rem}\label{rem:expandplethysm} 
It follows from \cite[Proposition 3.5.10]{HaimanSurvey} that 
$$
s_\mu \left[ \frac{Z}{(1 - t)} \right] = t^{n(\mu)} H_\mu(t)^{-1} \tilde{H}_{\mu}(Z;t,t^{-1}).
$$
Therefore 
$$
\left\langle s_\mu \left[ \frac{Z}{(1 - t)} \right], s_\lambda(z) \right\rangle = t^{n(\mu)} H_\mu(t)^{-1} \tilde{K}_{\lambda,\mu}(t,t^{-1}) = H_{\mu}(t)^{-1} K_{\lambda,\mu}(t,t).
$$
In the special cases $\lambda = (m)$ and $\lambda = (1^m)$ we have
\beq{eq:trivDelta}
\left\langle s_\mu \left[ \frac{Z}{(1 - t)} \right], s_{(m)}(z) \right\rangle = t^{n(\mu)}  H_\mu(t)^{-1},
\eeq
and
\beq{eq:signDelta}
\left\langle s_\mu \left[ \frac{Z}{(1 - t)} \right], s_{(1^m)}(z) \right\rangle = t^{n(\mu')}  H_{\mu'}(t)^{-1}.
\eeq
\end{rem}

\subsection{Combinatorics of $r$-cores and quotients}
As previously noted, $e_{\lambda,\mu}(v) = 0$ unless $\lambda$ and $\mu$ have the same $r$-core. We focus on the block of $\catO{m}{k}$ labeled by the $r$-core $(0)$. In this case it is shown in \cite{LTLR} that the numbers $e_{\lambda,\mu}(1)$ can be expressed in terms of Littlewood-Richardson coefficients.

In order to do this we require the notions of $r$-quotients and $r$-signs. Let $P_m = \Z^m$ be the weight lattice for $GL_m(\C)$ and $P_m^+$ the set of integral dominant weights. Write $\Pn_m = \{ \lambda \, | \, \lambda_1 \ge \ds \ge \lambda_m \ge 0 \} \subset P_m^+$, so that $\Pn_m$ is identified with the set of all partitions with at most $m$ parts. The extended affine Weyl group $\affs_m := P_m \rtimes \s_m$ acts on $P_m$ in the natural way. Let $\affs_m(r)$ denote the same group but now acting on $P_m$ by
$$
(s,v) \cdot \lambda = s \cdot \lambda + r v, \quad \forall \, (s,v) \in \s_m \rtimes P_m = \affs_m(r), \lambda \in P_m.
$$
If $\mu \in \Pn_m$ then the entries of $\mu + \rho$ are a set of $\beta$-numbers for $\mu$ in the sense of \cite[\S 2.7]{JK}. The result \cite[Lemma 2.7.13]{JK} says that if there exists some $1 \le i \le m$ such that $0 \le (\mu + \rho)_i - r \neq (\mu + \rho)_j $ for all $j \neq i$ then $(\mu_1, \ds, \mu_i - r, \ds ) + \rho = s \cdot (\lambda + \rho)$ for some $s \in \s_m$ and $\lambda \in \Pn_m$. Moreover, the partition $\lambda$ is a partition obtained from $\mu$ by removing an $r$-rim-hook. This implies that $\lambda, \mu \in \Pn_m$ have the same $r$-core if and only if $\mu + \rho \in \affs_m(r) \cdot (\lambda + \rho)$. Write $\core(\mu) \in \Pn_m$ for the $r$-core of $\mu$.

\begin{defn}
Let $\mu$ be a partition with at most $m$ parts. There exists a unique $s \in \s_m$ and $\lambda \in \Nat^m$ such that 
$$
\mu + \rho = s \cdot ( \core(\mu) + \rho) + r \cdot \lambda.
$$
For each $j \in \{0, \ds , r-1 \}$, the sub-sequence of $\lambda$ consisting of those $\lambda_i$ such that $s \cdot ( \core(\mu) + \rho)_i \equiv j - m \mod r$ defines a partition $\mu^{(j)}$. 
\begin{itemize}

\item The $r$-multi-partition $\quo(\mu) := (\mu^{(0)}, \ds, \mu^{(r-1)})$ is called the \textit{$r$-quotient} of $\mu$. 

\item The \textit{$r$-sign} of $\mu$ is $\varepsilon_r(\mu) := (-1)^{\ell(s)}$. 

\item We define
$$
\Rcore{r}{m} := \{ \mu \vdash m \, | \, \mu + \rho \in \affs_m(r) \cdot \rho \}
$$
to be the set of all partitions of $m$ with $r$-core $(0)$.
\end{itemize}
\end{defn}

\begin{rem}\label{rem:rdividescontent}
Let $\mu$ be a partition whose $r$-core is $(0)$. If $r$ is odd then $r$ divides $n(\mu) - n(\mu')$ and if $r$ is even then $r/2$ divides $n(\mu) - n(\mu')$. To see this, consider an $r$-rim-hook. Its content, starting from the top left box and going down to the bottom right box is a sequence of integers $i,i-1, i-2, \ds, i - r + 1$. Therefore the sum of its content is $r i - \frac{r(r-1)}{2}$. The claim now follows from the fact, as shown in the proof of Lemma \ref{lem:kappacalculation}, that $n(\mu') - n(\mu)$ is the content sum of $\mu$. In particular, this implies that $\kappa(k,\mu) \in \halfZ$ for all $\mu$ with $r$-core $(0)$. 
\end{rem}

\subsection{Littlewood-Richardson coefficients}
Let $\underline{\mu} = (\mu^{(0)}, \dots, \mu^{(r-1)})$ be an $r$-multi-partition. The Littlewood-Richardson coefficients are defined by
$$
c_{\underline{\mu}}^\lambda := \langle s_{\mu^{(0)}} \cdots s_{\mu^{(r - 1)}}, s_\lambda \rangle = [V_{\mu^{(0)}} \o \cdots \o V_{\mu^{(r-1)}} : V_\lambda].
$$
For $n \in \Nat$ and $k \in \Q_+$, define  
\beq{eq:definitionGpoly}
\Gpoly{k}{n}(\lambda,\nu;t) := \sum_{\mu  \in \Rcore{r}{r n}} \varepsilon_r(\mu) c_{\quo(\mu)}^\lambda t^{\kappa (k,\mu)} H_\mu(t)^{-1} K_{\nu,\mu}(t,t),
\eeq 
where $\lambda$ is a partition of $n$, $r$ is the denominator of $k$ and $\nu$ is a partition of $nr$.

\begin{prop}\label{prop:gpoly}
Let $k \in \Q_{+}$ and $\lambda \vdash n$. Let $r$ be the order of $\exp (2 \pi \sqrt{-1} k)$. Then 
\beq{eq:charLr2}
\cchi^{\adh \times \s_{m}}(L(r \lambda)) = \sum_{\nu \vdash m} \Gpoly{k}{n}(\lambda,\nu;t) \cdot \chi_\nu,
\eeq
where $m := nr$.
\end{prop}

\begin{proof}
Assume first that $r > 1$. The $r$-core of $r \lambda$ is $(0)$. Therefore $e_{r\lambda,\mu}(1) = 0$ unless $\mu + \rho \in \affs_m(r) \cdot \rho$ and $\mu \vdash m$. By making use of the Frobenius morphism on quantum groups, it is shown in \cite[Theorem 3.5]{LTLR} that
$$
e_{r\lambda,\mu}(1) = \varepsilon_r(\mu) c_{\quo(\mu)}^\lambda, \quad \forall \mu + \rho \in \affs_m(r) \cdot \rho, \, \mu \vdash m.
$$
Therefore
\beq{eq:charLr1}
\cchi^{\adh \times \s_m}(L(r \lambda)) = \sum_{\mu  \in \Rcore{r}{m}} \varepsilon_r(\mu) c_{\quo(\mu)}^\lambda t^{\kappa(k,\mu)} s_{\mu} \left[ \frac{Z}{(1 - t)} \right].
\eeq
Combine equation (\ref{eq:charLr1}) with the main equation in remark \ref{rem:expandplethysm} to get equation (\ref{eq:charLr2}).

If $r = 1$ then the combinatorics of the canonical basis does not make sense. However, in this case the Hecke algebra of type $\mathbf A$ at $q = \exp (2 \pi i k) = 1$ is just the group algebra and hence semi-simple. Therefore results of \cite{GGOR} imply that both category $\mathcal{O}$ and the quantized Schur algebra are also semi-simple. At the same time, $c^\lambda_{\quo(\mu)} = c^{\lambda}_\mu = \delta_{\lambda,\mu}$ for all $\lambda, \mu \vdash m$. Therefore equation (\ref{eq:charLr2}) is still valid.
\end{proof}

The definition of $\Gpoly{k}{n}(\lambda,\nu;u)$ as the character of a $W$-isotypic component of an alternating sum of standard modules shows that $\Gpoly{k}{n}(\lambda,\nu;u)$ can be expressed as a rational function whose denominator is $\prod_{i = 1}^m (1 - u^i)$. It also implies that, as a power series in $u$, $\Gpoly{k}{n}(\lambda,\nu;u) \in u^{\alpha} \Nat [[u]]$. Note that when each partition in $\quo(\mu)$ has at most one part, $c_{\quo(\mu)}^\lambda = K_{\lambda,\mu}$ the Kostka number. In the simplest case, where $\lambda = (1)$, one can argue as in the proof of \cite[Proposition 2.5.3]{HaimanConjectures} to show that

\begin{lem}\label{prop:GordonDiagonal}
Let $k = l/m$ with $l > 0$ and $\hcf (l,m) = 1$ and choose $\mu \vdash m$. Then 
$$
\Gpoly{k}{1}((1),\mu;t) = \frac{t^{\gs}}{(1 - t^l)} \cdot s_{\mu} \left[ \frac{1 - t^l}{1 - t} \right]
$$
where $\gs := \frac{m + l -ml}{2}$. 
\end{lem}

\subsection{The Calaque-Enriquez-Etingof functor}\label{sec:CEEfunctor}
Let $\one$ denote the identity matrix, viewed as a basis of the centre of $\g$ so that $\g = \C \cdot \one \oplus \mathfrak{sl}_n$. Let $\{ e_{\alpha} \}_{\alpha}$ be an orthonormal basis of $\mathfrak{sl}_n$ with respect to the trace form and write $\{ x_{\alpha} \}_{\alpha}$ for the basis of $\g^*$ defined by $x_\alpha(e_\beta) = \delta_{\alpha,\beta}$. The vectorial representation of $\g$ will be denoted $V$. Let $\det$ denote the determinant representation for $GL_n$ and fix $m = n \cdot r$ for some $r \in \Nat$. We denote by $\Lmod{\dd_G(\g)}$ the category of $G$-equivariant (in the sense of \cite[Definition 3.1.3]{KashCIME}), coherent $\dd(\g)$-modules. The rational Cherednik algebra $\Ch{m}{k}$ is isomorphic to $\mathcal{H}_m(k) \o \dd(\mathbb{A}^1)$, where $\mathcal{H}_m(k)$ is the rational Cherednik algebra of type $\mathbf{A}_{m-1}$. Similarly, $\dd(\g) = \dd(\mathfrak{sl}_n) \o \dd(\mathbb{A}^1)$. This factorization allows us to extend the functor defined by Calaque, Enriquez and Etingof for $\mathfrak{sl}_n$ to a functor for $\g$. Define $x_{\one} := \frac{1}{\sqrt{n}} \sum_{i = 1}^n x_{i,i}$ and $\pa_{\one} := \frac{1}{\sqrt{n}} \sum_{i = 1}^n \pa_{i.i}$ so that $\{ x_{\alpha} \}_{\alpha} \cup \{ x_{\one} \}$ is an orthonormal basis of $\g^*$. Combining \cite[Proposition 6.1]{CEE} and \cite[Proposition 8.1]{CEE} gives:

\begin{lem}\label{lem:CEEfunctor}
Let $M \in \Lmod{\dd_G(\g)}$. The formulae
$$
\frac{1}{\sqrt{m}} \sum_{i = 1}^m x_i \ \mapsto \ x_{\one} \o \id, \quad x_i - x_{i+1} \ \mapsto \ \frac{1}{\sqrt{r}} \sum_{\alpha} x_{\alpha} \o (e^{(i)}_{\alpha} - e^{(i+1)}_{\alpha}), \ \forall 1 \le i \le m-1, 
$$
$$
\frac{1}{\sqrt{m}} \sum_{i = 1}^m y_i \mapsto \pa_{\one} \o \id, \quad y_i - y_{i+1} \mapsto \frac{1}{\sqrt{r}} \sum_{\alpha} \pa_{\alpha} \o (e^{(i)}_{\alpha} - e^{(i+1)}_{\alpha}), \ \forall 1 \le i \le m-1,
$$
and $s_{ij} \mapsto s_{ij}, \ \forall 1 \le i \neq j \le m$ define an action of $\Ch{m}{\frac{1}{r}}$ on $(M \o V^{\o m} \o \det^{-r})^{G}$.
\end{lem}

Following \cite{CEE}, define 
$$
F_m \, : \, \Lmod{\dd_G(\g)} \longrightarrow \Ch{m}{k}\text{-}\mathsf{Mod}, \quad M \mapsto (M \o V^{\o m} \o \mathrm{det}^{-r})^{G},
$$
where $k = \frac{1}{r}$. The functor $F_m$ is exact. 

The Fourier transform is an automorphism $\mathcal{F} : \Ch{m}{k} \rightarrow \Ch{m}{k}$ defined by $\mathcal{F}(x_i) = y_i$, $\mathcal{F}(y_i) = - x_i$ and $\mathcal{F}(w) = w$ for all $ w \in \s_m$. If $L$ is a left $\Ch{m}{k}$-module then $\mathcal{F}(L)$ is also a left $\Ch{m}{k}$-module with action $f \bullet l = \mathcal{F}(f) \cdot l$, for $f \in \Ch{m}{k}$ and $l \in L$. Define $F^*_m = \mathcal{F} \circ F_m$. 

Let $\mathcal{O}_\lambda$ be the set of all nilpotent elements with Jordan type $\lambda$ e.g. $\mathcal{O}_{(n)}$ is the orbit of regular nilpotent elements and $\mathcal{O}_{(1^n)} = \{ 0 \}$. As is explained in (\ref{sec:mnil}), the $\dd(\g)$-module $\mm_{\lambda}$ the unique simple $\dd(\g)$-module supported on the closure of $\mathcal{O}_\lambda$. Let $\Lmod{\dd_G(\N)}$ be the category of $G$-equivariant, coherent $\dd(\g)$-modules set-theoretically supported on the nilpotent cone $\N$. The simple objects in $\Lmod{\dd_{G}(\N)}$ are $\mm_{\lambda}$ such that $\lambda \vdash n$ and it is shown in \cite[Theorem 9.1]{CEE}. 

\begin{thm}[\S 9, \cite{CEE}]\label{thm:Fmainproperties}
Let $k = \frac{1}{r}$, then the functor $F^*_m$ restricts to a functor 
$$
F_m^* : \Lmod{\dd_G(\N)} \longrightarrow \catO{m}{k}
$$
such that $F_m^* (\mm_\lambda) = L( r \lambda)$.
\end{thm} 

\section{The Harish-Chandra module and Cherednik algebras}\label{sec:Dmod}

For subsection (\ref{sub:hc}) to (\ref{section:Fourier}) let $G$ be any connected complex reductive group as in (\ref{sec:groupactions}).

\subsection{The \hc module $\mm$}\label{sub:hc}

Harish-Chandra defined the \textit{radial parts} map $\rad : \dd(\g)^G  \rightarrow \dd(\t)^{W}$ such that $\rad$ restricts to the Chevalley map $\C[\g]^G \stackrel{\sim}{\rightarrow} \C[\t]^W$, $P(x) \mapsto P(t)$, and $\sym (\g)^G \stackrel{\sim}{\rightarrow} \sym (\t)^W$, $P(\partial_x) \mapsto P(-\partial_t)$, respectively. Let $\ad : \g \rightarrow \dd(\g \times \t)$ be the map induced by the adjoint action of $G$ on $\g$. 

\begin{defn}
The \hc module $\mm$ is the cyclic left $\dd(\g \times \t)$-module, generated by $u_0$, such that 
$$
\ad(\g) \cdot u_0 = 0, \quad (f \o 1 - 1 \o \rad(f)) \cdot u_0 = 0, \quad \forall f \in \dd(\g)^G.
$$
\end{defn}

The above definition of the \hc module is not the same as the definition given in \cite{HottaKashiwara}. However, it is shown in \cite[Remark 4.1.2]{Iso} that the two definitions are equivalent. It is shown in \cite[Theorem 4.2]{HottaKashiwara} that $\mm$ is a simple, holonomic module. Let $\Cs$ act on $\g \times \t$ by dilations. The action of $G \times W \times \Cs$ on $\g \times \t$ lifts to an action of $G \times W \times \Cs$ on $\dd(\g \times \t)$. The induced action of $G \times W \times \Cs$ on $\ggr^{\mathrm{ord}} \dd(\g \times \t) \simeq \gg \times \tt$ is the specialization $t \mapsto q^{-1}$ of the $H \times W$-action defined in (\ref{sec:groupactions}). Since the ideal defining $\mm$ is a homogeneous $G \times W$-stable ideal, $G \times W \times \Cs$ also acts on $\mm$ making it a quasi-$G \times W \times \Cs$-equivariant $\dd(\g \times \t)$-module (in the sense of \cite[Definition 3.1.3]{KashCIME}). 

As explained in \cite[\S 2.5]{Iso}, there is a canonical Hodge filtration on the \hc module. It is then shown in \cite[Theorem 1.3.3]{Iso} that there is a natural isomorphism 
\beq{eq:hodge}
\ggr^\hodge\mm \simeq \ \varphi_* \mathcal{O}_{\xx_{\norm}}
\eeq 
of $\mathcal{O}_{\gg \times \tt}$-modules, where $\varphi$ is the composition $\xx_{\norm} \rightarrow \xx \hookrightarrow \gg \times \tt$. The Hodge filtration defines a $\Z$-grading on $\ggr^\hodge\mm$. Moreover, since the Hodge filtration is canonical, each filtered piece $F_k^{\hodge} \mm$ is $G \times W \times \Cs$-stable and hence $\ggr^\hodge\mm$ is a $H \times W$-module. The isomorphism (\ref{eq:hodge}) is $H \times W$-equivariant. 

\subsection{The sheaf $\mmnil$}\label{sec:mnil}
Let $j : \g \hookrightarrow \g \times \t$, $x \mapsto (x,0)$, be the inclusion map and write $\mmnil := j^* \mm$. Let $\Irr (W)$ be a set parameterizing the isomorphism classes of irreducible $W$-modules. For each $\mu \in \Irr (W)$ we write $\chi_\mu$ for the corresponding irreducible $W$-module. Since $W$ acts trivially on $\g$, there is a decomposition
$$
\mmnil = \bigoplus_{\mu \in \Irr (W)} \mm_\mu \o \chi_\mu.
$$
It is shown in \cite[Theorem 5.3]{HottaKashiwara} that each $\mm_\mu$ is a simple, holonomic $\dd(\g)$-module supported on the closure of a nilpotent orbit in $\g$. Moreover, the modules $\mm_\mu$ are pairwise non-isomorphic. Therefore there exists an orbit $\mathcal{O}_\mu$ and irreducible, $G$-equivariant local system $L_\mu$ on $\mathcal{O}_\mu$ such that $\mm_\mu$ corresponds under the Riemann-Hilbert correspondence to the perverse sheaf $\IC (\mathcal{O}_\mu,L_\mu)$. The rule $\mu \mapsto (\mathcal{O}_\mu,L_\mu)$ is one incarnation of Springer's correspondence. When $G = GL_n$ it follows from \cite[Proposition 3.4.14]{HaimanSurvey} that the correspondence just sends the partition $\mu$ to the orbit $\mathcal{O}_\mu$ of nilpotent matrices with Jordan type $\mu$ and the trivial local system $L_0$ on $\mathcal{O}_\mu$. Since each piece $F^\hodge_k \mm$ of the Hodge filtration on $\mm$ is a $\C[\t]$-module there is a canonical quotient filtration $F^\quo_{\bullet} \mmnil$ on $\mmnil$. Each piece of this filtration is $H \times W$-stable. This implies that each $\mm_{\mu}$ inherits from $\mm$ a canonical filtration, which we will also refer to as the quotient filtration.

\subsection{Euler grading}\label{section:Fourier}
For a vector space $V$, let $\eul_V$ denote the first order differential operator in $\dd(V)$ with constant term zero corresponding to the Euler vector field along $V$. A $\dd(V)$-module $\Lsimp$ is said to be \textit{monodromic} if the action of $\eul_V$ is locally finite. The Euler operator defines a grading on monodromic $\dd(V)$-modules. The differential of the action of $\Cs$ on $\dd(\g \times \t)$ induced from the action by dilations on $\g \times \t$ is given by $1 \mapsto \eul_{\g\times \t}$. Let $K$ be a reductive group acting on $V$. Recall (c.f. \cite[Definition 3.1.3]{KashCIME}) that the $\dd(V)$-module $\mathcal{L}$ is said to be $K$-equivariant there is an action of $K$ on $\mathcal{L}$ such that the differential of this action coincides with the composition of the morphism $\ad : \Lie (K) \rightarrow \dd(V)$ with the action map $\dd(V) \rightarrow \End_{\C}(\mathcal{L})$.

For a $\halfZ$-graded vector space $L$ we let $L[i]$ denote the space $L$ with grading shifted by $i$ so that $L[i]_k = L_{k - i}$. 

\begin{prop}\label{lem:stronglyequivariant}
Choose $\mu \in \Irr (W)$ and fix $\gs := -\frac{\dim (\g + \t)}{2}$. Then the quasi-$G \times \Cs$-equivariant $\dd(\g)$-module $\mm_{\mu}[\gs]$ is actually $G \times \Cs$-equivariant.
\end{prop}

\begin{proof}
It follows from the definition of $\mm_\mu$ given above that it is a $G$-equivariant $\dd(\g)$-module. Therefore we just need to show that it is also $\Cs$-equivariant. 

First we show that $\mm[\gs]$ is a $\Cs$-equivariant $\dd(\g \times \t)$-module. For this we just have to show that if $\bar{y} \in \mm[\gs]$ is homogeneous, $\lambda \cdot \bar{y} = \lambda^n \bar{y}$ for some $n \in \Z$ then $\eul_{\g \times \t} \cdot \bar{y} = n \bar{y}$. To begin with, we consider the action of $\eul_{\g \times \t}$ on $u_0$. The Killing form on $\g$ induces a non-degenerate bilinear pairing on $\g \times \t$. Let $\{ x_\alpha \}$, respectively $\{ y_\beta \}$, be an orthonormal basis on $\g$, respectively $\t$. Since $[\pa_\alpha^2, x_\alpha^2] = 4 x_\alpha\pa_\alpha + 2$, the elements
$$
E_1 = \frac{1}{2} \sum_{\alpha} x_\alpha^2, \quad H_1 = \sum_{\alpha} x_\alpha \pa_{x_\alpha} + \frac{\dim \g}{2}, \quad F_1 = \frac{-1}{2} \sum_{\alpha} \pa_{x_\alpha}^2,
$$
and 
$$
E_2 = \frac{1}{2} \sum_{\beta} y_\beta^2, \quad H_2 = \sum_{\beta} y_\beta \pa_{y_\beta} + \frac{\dim \t}{2}, \quad F_2 = \frac{-1}{2} \sum_{\beta} \pa_{y_\beta}^2,
$$
form $\mathfrak{sl}_2$-triples in $\dd(\g)^{G}$ and $\dd(\t)^{W}$ respectively. Following \cite[Lemma 7.1.1]{HottaKashiwara} and using the fact that $\rad(E_1) = E_2$, $\rad(F_1) = F_2$ we get
$$
0 = [E_1 \o 1 - 1 \o E_2,F_1 \o 1 - 1 \o F_2] \cdot u_0 = (H_1 + H_2) \cdot u_0 = (\eul_{\g \times \t} - \gs) \cdot u_0.
$$
Now choose a homogeneous lift $y$ in $\dd(\g \times \t)$ of $\bar{y} = y \cdot u_0$. Then $\lambda \cdot y = \lambda^{n - \gs} y$ and hence
$$
\eul_{\g \times \t} \cdot \bar{y} = [\eul_{\g \times \t}, y] \cdot u_0 - \gs \cdot u_0 = n \cdot \bar{y}.
$$
It is shown in \cite[Proposition 4.8.1]{HottaKashiwara} that $\mmnil = \mm / \t \cdot \mm$. Since $\eul_{\g \times \t} \cdot \t \cdot \mm \subset \t \cdot \mm$, the operator $\eul_{\g \times \t}$ acts on $\mmnil$. Moreover, the fact that $\eul_{\t} \cdot \mm \subset \t \cdot \mm$ shows that the action of $\eul_{\g \times \t}$ on $\mmnil$ equals the action of $\eul_{\g}$ on $\mmnil$. Thus $\mmnil[\gs]$ is a $\Cs$-equivariant $\dd(\g)$-module, as required.
\end{proof}

\subsection{} From now on, we return to the case of $G=GL_n$, so $W=\s_n$.

Let $P_n$ denote the weight lattice of $GL_n(\C)$ and $Q_n$ the root lattice. Recall that $Q^+_n = P^+_n \cap Q_n$. Denote by $\Jset$ the set of all weights $\nu \in P^+_n$ such that $\nu_n = 1$ and $| \nu | \equiv 0 \mod n$. There is a natural bijection $( - )^\dagger : Q^+_n \stackrel{\sim}{\longrightarrow} \Jset$, $\mu^\dagger := \mu + ((1 - \mu_n)^n)$. For brevity write $\mathcal{G} (\lambda, \nu;t) := \Gpoly{n / |\nu|}{n}(\lambda,\nu;t)$. 

\begin{thm}\label{prop:charhc1}
Let $\lambda \vdash n$, then
\beq{eq:charhc1}
\cchi^{G \times \eul_{\g}}(\mm_\lambda;q) = q^{-\frac{\dim \g}{2}} \sum_{\mu \in Q^+_n} \mathcal{G} (\lambda, \mu^{\dagger};q^{-1}) \cdot s_{\mu}(z).
\eeq
\end{thm}

\begin{proof}
For $\mu \in P_n^+$, let $(\mm_{\lambda})_{\mu} = (\mm_{\lambda} \o V^*_{\mu})^G$ be the $\mu$-isotypic component of $\mm_{\lambda}$. As a $G$-equivariant $\dd(\g)$-module, $\mm_{\lambda}$ is a direct summand of a certain quotient of $\dd(\g)$. This implies that $(\mm_\lambda)_{\mu} \neq 0$ only if $\mu \in Q^+_n$. Write 
$$
\cchi^{G \times \eul_{\g}}(\mm_\lambda;q) = \sum_{\mu \in Q^+_n}  \cchi^{\eul_{\g}}((\mm_\lambda)_{\mu};q) \cdot s_{\mu}(z).
$$
By Theorem \ref{thm:Fmainproperties} and Schur-Weyl duality, 
$$
F^*_{nr}(\mm_\lambda) = \bigoplus_{\mu} \ (\mm_\lambda)_{\mu} \o \chi_{\mu + (r^n)} = L(r \lambda),
$$ 
where the sum is over all $\mu \in Q^+_n$ such that $\mu_i \ge -r$ for all $i$. Therefore, if we fix $\mu \in Q^+_n$ and let $r = 1 - \mu_n$, then 
$$
(\mm_\lambda)_{\mu} = (L(r \lambda) \o \chi_{\mu^\dagger})^{\s_m}.
$$
This is an equality of graded vector spaces and we just need to match up the grading. The Euler operator $\eul_\g$ acts on the left hand side and the operator $\adh$ acts on the right. Using the fact that $\Ch{m}{k} = \mathcal{H}_m(k) \o \dd(\mathbb{A}^1)$ and that $\dd(\g) = \dd(\mathfrak{sl}_n) \o \dd(\mathbb{A}^1)$, together with the fact that $\mathcal{F}(\adh) = -\adh$, it follows from \cite[Proposition 8.7]{CEE} that we have an equality of operators $\eul_\g + \frac{\dim \g}{2} = -\adh$ on the space $(\mm_\lambda)_{\mu}$. This implies that 
$$
\cchi^{\eul_\g}((\mm_\lambda)_{\mu};q) = q^{- \frac{\dim \g}{2}} \cchi^{\adh}((L(r \lambda) \o \chi_{\mu^\dagger})^{\s_m};q^{-1}).
$$
Equation (\ref{eq:charhc1}) now follows from equation (\ref{eq:charLr2}). 
\end{proof}

\begin{rem}
If the Lie group $GL_n(\C)$ is replaced by $SL_n(\C)$ then there exist non-trivial irreducible $G$-equivariant local systems on some nilpotent orbits $\mathcal{O}_\lambda$. Therefore the number of simple modules in the category $\dd_{SL_n}(\N)$ is greater. One can repeat the above arguments to get a formula for the $G \times \eul_\g$-character of these simple modules in terms of the polynomials $\mathcal{G} (\lambda, \mu^{\dagger};q^{-1})$ as in Theorem \ref{prop:charhc1}. The details are left to the interested reader.
\end{rem}  

\subsection{Relation to Hesselink's character formula}

Let $\sym \g$ denote the algebra of constant coefficient differential
operators in $\dd(\g)$. When $\lambda = (1^n)$, $\mm_{(1^n)}$ is the
unique simple $\dd(\g)$-module supported on $\{ 0 \}$. This module can
be naturally identified with $\sym \g$. Using Kostant's Theorem,
Hesselink \cite{Hesselink} showed that
$$
\cchi^{G \times \Cs}(\sym \g;q) = \frac{1}{\prod_{i = 1}^n (1 - q^{-i})}
\sum_{\mu \in Q^+_n} K_{\mu,0}(q^{-1}) \cdot s_\mu(z),
$$
where the $\Cs$-action corresponds to the grading that places $\sym^k
\g$ in degree $-k$. This differs from the grading on $\mm_{(1^n)}$
coming from $\mm$ by a shift. The fact that $\sym \g$ is the cyclic
$\dd(\g)$-module generated by $v_0$ and satisfying the relation $\g^*
\cdot v_0 = 0$ implies that $\eul_\g \cdot v_0 = - \dim \g \cdot v_0$
and hence $\cchi^{G \times \Cs}(\sym \g;q) = q^{\dim \g} \cdot \cchi^{G
\times \eul_\g}(\sym \g;q)$. Comparing this equation with equation
(\ref{eq:charhc1}) and remebering that $t = q^{-1}$ produces the
identity
\beq{eq:compareHesselink}
\mathcal{G}((1^n),\mu^\dagger; t) = \frac{t^{\frac{\dim \g}{2}}
K_{{\mu},0}(t)}{\prod_{i = 1}^n (1 - t^i)}.
\eeq

\subsection{} Recall (\ref{tx}) the map $u : \zz_{\norm} \rightarrow
\gg$
and the coherent sheaf $\rrnil$ on $\zz_{\norm}$, see \S\ref{rnilsec}.

\begin{prop}\label{prop:mmnilfilter}
There is a canonical $H \times W$-equivariant isomorphism 
\beq{eq:nilhodgeiso}
\ggr^\quo \mmnil \simeq u_* \ \rrnil,
\eeq
of coherent sheaves on $\gg$. 
\end{prop}

\begin{proof}
The statement of the proposition is proved as part of the proof of \cite[Theorem 2]{GordonMacdonald}. Let $y_1, \ds, y_n$ be homogeneous algebraically independent generators of $\C[\t]$. Then it is shown in \cite[Claim 1]{GordonMacdonald} that $y_1, \ds, y_n$ form a regular sequence for $\ggr^\hodge \mm$. Equipped with this fact, the proof of \cite[Theorem 4.7]{KashAMS} shows that $\ggr^\quo \mmnil \simeq \ggr^\hodge \mm \ / \ \t \cdot \ggr^\hodge \mm$. Now the proposition follows from the isomorphism (\ref{eq:hodge}).
\end{proof}

Note that the proof of Proposition \ref{lem:stronglyequivariant} implies
that the modules $\mm$ and $\mm_\mu$ are monodromic. 
 Proposition \ref{prop:mmnilfilter} implies that there is a $H$-equivariant isomorphism of sheaves $\ggr^\hodge\mm_{\mu} \simeq \rrnil_{\mu}$ for each $\mu \in \Irr (W)$. 

We obtain the following result.

\begin{cor}\label{cor:Xnilchar}
Let $\lambda$ be a partition of $n$. The graded $G$-character of $\rrnil_\lambda$ is 
\beq{eq:charnil}
\cchi^{G \times \Cs}( \rrnil_\lambda;q ) = q^{\frac{\dim \t}{2}} \sum_{\mu \in Q^+_n} \mathcal{G} (\lambda,\mu^{\dagger};q^{-1}) \cdot s_{\mu}(z).
\eeq
\end{cor}

\begin{proof}
Proposition \ref{lem:stronglyequivariant} and Theorem \ref{prop:charhc1} imply that 
$$
\cchi^{G \times \Cs}(\mm_\lambda;q) = q^{\frac{\dim \t}{2}} \sum_{\mu \in Q^+_n} \mathcal{G} (\lambda, \mu^{\dagger};q^{-1}) \cdot s_{\mu}(z).
$$
Therefore equation (\ref{eq:charnil}) follows from the $G \times \Cs$-equivariant isomorphism (\ref{eq:nilhodgeiso}).
\end{proof} 

\begin{cor}
The graded $\s_n$-character of $(\rrnil)^G$ is 
\beq{eq:charnilinv}
\cchi^{\Cs \times \s_n}((\rrnil)^G;q) = \sum_{\lambda \vdash n} q^{- n(\lambda)}H_{\lambda}(q^{-1})^{-1} \cdot \chi_{\lambda}.
\eeq
\end{cor}

\begin{proof}
Since 
$$
\cchi^{\Cs \times \s_n}((\rrnil)^G;q ) = \sum_{\lambda \vdash n} \cchi^{\Cs}( (\rrnil_\lambda)^G;q ) \cdot \chi_{\lambda},
$$
Corollary \ref{cor:Xnilchar} implies that 
$$
\cchi^{\Cs \times \s_n}((\rrnil)^G;q ) = q^{\frac{\dim \t}{2}} \sum_{\lambda \vdash n} \mathcal{G} (\lambda, (1^n);q^{-1}) \cdot \chi_{\lambda},
$$
where we have used the fact that $(0)^\dagger = (1^n)$. In this situation $\mathcal{G} (\lambda, (1^n);t) = \Gpoly{1}{n} (\lambda, (1^n);t)$ and the proof of Proposition \ref{prop:gpoly} (in the case $r = 1$) shows that 
$$
\mathcal{G} (\lambda, (1^n);t) = t^{\kappa(1,\lambda)} \left\langle s_{(1^n)}(z), s_{{\lambda}} \left[ \frac{Z}{(1 - t)} \right] \right\rangle = t^{\kappa(1,\lambda) + n(\lambda')} H_{\lambda'}(t)^{-1}.
$$
Noting that $\kappa(1,\lambda) = \frac{n}{2} + n(\lambda) - n(\lambda')$ and the fact that $H_{\lambda'}(t) = H_{\lambda}(t)$, we get the required formula.
\end{proof}

\begin{rem}\label{rem:reoccuringdream}
It follows from \cite[Theorem 1.6.1]{Iso} that $\Gamma(\zz_{\norm},\rrnil)^G = \C[\t]$, where the action of $\Cs$ on $\t$ is expanding. This agrees with equation (\ref{eq:charnilinv}) since $q^{- n(\lambda)}H_{\lambda}(q^{-1})^{-1}$ is just the graded character of $(\C[\t] \o \chi_{\lambda})^{\s_n}$, which can be seen from the corresponding formula for the fake polynomial \cite[Theorem 3.2]{Stembridge}. We will give yet another derivation of the character formula (\ref{eq:charnilinv}) in (\ref{rem:yetanotherderivation}).
\end{rem}

We now have, at least for those weights $\mu^\star \ | \ \mu \in
 \Pnone$, two different expressions for the character of the various
 $G$-isotypic components of $\rr$. The first expression comes from the
geometry of the isospectral commuting variety and the Hilbert
scheme. The second expression comes from the theory
of Cherednik algebras.
 Comparing these formulae via the
Calaque-Enriquez-Etingof functor produces 
the following rather interesting and nontrivial identities.

\begin{thm}\label{thm:identities}
Let $\mu \in \Pnone$ and $\lambda \vdash n$, then
$$
\Phi_\mu \left[ B_\lambda(q,q^{-1}) \right] = q^{\frac{\dim \t}{2} + n(\lambda)} H_\lambda (q^{-1}) \ \mathcal{G} ( \lambda,(\mu^\star)^\dagger;q^{-1})
$$
and
$$
K_{\mu^\star,0}(t) = \Phi_{\mu} \left[ \frac{1 - t^n}{1 - t} \right].
$$
\end{thm}

\begin{proof}
Comparing equation (\ref{eq:specialisePhi}) with equation (\ref{eq:charnil}) produces the first equality. Then comparing this with equation (\ref{eq:compareHesselink}) and noting that $B_{(1^m)}(q,q^{-1}) = (1 - q^{-n}) / (1 - q^{-1})$ and $H_{(1^m)}(q^{-1}) = \prod_{i = 1}^n (1 - q^{-i})$ produces the second identity. 
\end{proof}

\begin{rem}\label{rem:yetanotherderivation}
If we take $\mu = (0)$ in equation (\ref{eq:specialisePhi}) then we get yet another proof of the identity $\cchi^{\Cs}((\rrnil_{\lambda})^G) = q^{-n(\lambda)} H_{\lambda}(q^{-1})^{-1}$ (c.f. remark \ref{rem:reoccuringdream}).
\end{rem}

\subsection{Filtrations on simple $\Ch{m}{k}$-modules}

An interesting consequence of the fact that the simple $\dd(\g)$-modules $\mm_\mu$ have a canonical filtration coming from the Hodge filtration on $\mm$ is that the simple $\Ch{m}{k}$-modules $L( r \lambda)$ also have a canonical filtration. Define a filtration on $\Ch{m}{k}$ by putting the $y_i$ and group elements in degree zero and the $x_i$ in degree one. Then the associated graded of $\Ch{m}{k}$ is the skew group ring $\C[\mathfrak{h}^* \times \mathfrak{h}] \rtimes \s_m$, where $\mathfrak{h}$ is an $m$-dimensional vector space such that the symbol of $x_i$ becomes a linear function on $\mathfrak{h}$ and $y_j$ a linear function on $\mathfrak{h}^*$. Now fix $r > 0$, $m = nr$ and $k = 1/r$. Recall that $F_{\bullet}^\quo \mm_\lambda$ denotes the filtration on $\mm_\lambda$ inherited from the Hodge filtration on $\mm$. Define $F_{\bullet}^\can L(r \lambda)$ by 
$$
F_{l}^\can L(r \lambda) := F^*_m(F_{l}^\quo \mm_\lambda).
$$
The quotient filtration on $\mm_\lambda$ is compatible with the order filtration on $\dd(\g)$. Bearing in mind that the Fourier transform swaps $x_{i}$ with $y_{i}$, it is then clear from Lemma \ref{lem:CEEfunctor} that the filtration $F_{\bullet}^\can L(r \lambda)$ is compatible with the filtration on $\Ch{m}{k}$. 

\begin{lem}
The $\mathcal{O}_{\mathfrak{h}^* \times \mathfrak{h}}$-module $\ggr^\can L(r \lambda)$ is coherent.
\end{lem}

\begin{proof}
The quotient filtration $F_{k}^{\quo} \mm_\lambda$ is good with respect to the order filtration on $\dd(\g)$. Therefore $\ggr (\mm_\lambda)$ is a finitely generated $\C[\gg]$-module. Since $G$ is reductive, Hilbert's Theorem (c.f. \cite[Zusatz 3.2]{Kraft}) implies that $( \ggr (\mm_\lambda) \o V^{\o m} \o \mathrm{det}^{-r})^G$ is a finitely generated $\C[\gg]^G$-module. The reductivity of $G$ also implies that  
$$
( \ggr (\mm_\lambda) \o V^{\o m} \o \mathrm{det}^{-r})^G \simeq \ggr^\can L(r \lambda).
$$
Since $\mm_\lambda$ is supported on the null cone $\N$, $\C[\g]^G$ acts locally nilpotently on $\ggr (\mm_\lambda)$. Hence $\ggr^\can L(r \lambda)$ is a finitely generated $\C[\g^*]^G$-module. The equation \cite[(43)]{CEE} shows that the action of $\C[\mathfrak{h}]^W$ on $\ggr^\can L(r \lambda)$ factors through a surjective morphism $\C[\mathfrak{h}]^W \twoheadrightarrow \C[\g^*]^G$. This implies that $\ggr^\can L(r \lambda)$ is a finitely generated $\C[\mathfrak{h}]^W$-module. 
\end{proof}

The isomorphism $F^*_m(\mm_\lambda) \simeq L(r \lambda)$ means that the $\C^*$-action on $\mm_\lambda$ defines a $\C^*$-action on $L(r \lambda)$. This is related to the grading coming from the operator $\adh$ by
$$
\cchi^{\C^*}(L(r \lambda);q) = q^{\frac{n}{2}} \cchi^{\adh}(L(r \lambda);q^{-1}).
$$
Since the filtration $F_{\bullet}^\quo \mm_\lambda$ respects the $G \times \Cs$-action on $\mm_\lambda$, the filtration $F_{\bullet}^\can L(r \lambda)$ respects the $\s_m \times \Cs$-action on $L(r \lambda)$. Define an action of $\cc$ on $\mathfrak{h}^* \times \mathfrak{h}$ by making the first copy of $\Cs$ act by dilations on $\mathfrak{h}^*$ and the second copy of $\Cs$ acts by dilations on $\mathfrak{h}$. Then $\ggr^\can L(r \lambda)$ is a $\s_m \times \cc$-equivariant $\mathcal{O}_{\mathfrak{h}^* \times \mathfrak{h}}$-module. 

\begin{prop}\label{prop:rcasimple}
The bigraded $\s_m$-character of $\ggr^\can L(r \lambda)$ is 
$$
\cchi^{\s_m \times \cc}(\ggr^\can L(r \lambda)) = \sum_{\stackrel{\mu \in \mathcal{I}_{n,r};}{\nu \le \lambda.}} \frac{\Phi_\mu \left[ B_\nu(q,t) \right] \cdot {\mathbf k}_{\lambda,\nu}(q,t) \cdot \chi_{\mu^\star + (r^n)}}{\prod_{x \in Y_{\mu}}(1 - t^{1 + l(x)} q^{- a(x)}) (1 - t^{- l(x)} q^{1 + a(x)})}
$$
where $\mathcal{I}_{n,r}$ is the set of all $\mu \in \Pnone$ such that $\mu_i \le r$ for all $i = 1, \ds, n-1$. 
\end{prop}

\begin{proof}
The compatibility of the Hodge filtration with the action of $H$ means that we have an equality of $W \times \cc$-modules
$$
\ggr^\can L(r \lambda) = (\ggr^\quo \mm_\lambda \o V^{\o m} \o \mathrm{det}^{-r})^G = (\rrnil_\lambda \o V^{\o m} \o \mathrm{det}^{-r})^G.
$$
Via Schur-Weyl duality, this implies that 
$$
\cchi^{\s_m \times \cc}(\ggr^\can L(r \lambda)) = \sum_{\mu} \cchi^{\cc}((\rrnil_\lambda)_{\mu^\star}) \cdot \chi_{\mu^\star + (r^n)},
$$
where the sum is over the set of all $\mu \in \Pnone$ such that $\mu^\star = \eta - (r^n)$ for some partition $\eta$ of $m$ with at most $n$ parts. One can check that this set is precisely $\mathcal{I}_{n,r}$. Now the equation follows from (\ref{eq:Phinil}).
\end{proof}

It would be interesting to have an equivalent definition of the filtration $F_{\bullet}^\can L(r \lambda)$ that does not involve the functor $F^*_m$.

\section{Appendix by Eliana Zoque: \\ $T$-orbits of \pnps}
\subsection{}
Let $\g=\sl_n$. It has been proved in \cite{PNP} that every principal nilpotent pair is associated to the Young diagram of a partition $\mu$ of size $n$. We consider an $n$-dimensional vector space $V$ with a basis indexed by the boxes in the Young diagram of $\mu$. Consider $V=\langle v_{i,j}\rangle$ where $v_{i,j}$ corresponds to the square in the $i$-th row and $j$-th column. If $(i,j)$ is outside the diagram we define $v_{i,j}=0$. Let $\mathbf{e}=(e_1,\,e_2)$ be the principal nilpotent pair of $V$ defined by $e_1v_{i,j}=v_{i+1,j},\,e_2v_{i,j}=v_{i,j+1}$.

The main goal of this Appendix is to prove the following result.
\begin{thm}\label{thmclosure}
$T\cdot\mathbf{e}$ is dense in $(\g_{1,0}\oplus\g_{0,1})\cap\mathfrak{C}$.
\end{thm}

Let
$\mathbf{x}=(x_1,x_2)\in(\g_{1,0}\oplus\g_{0,1})\cap\mathfrak{C}$. 
The pair $\mathbf{x}$ admits the same associated semisimple pair as $\mathbf{e}$, therefore by Lemma 5.7 in \cite{PNP} it is conjugate to a pair of the form $\bigoplus_k\mathbf{e}_{\lambda_k}$ for a certain collection of skew-diagrams $\mathbb{\lambda}=\{\lambda_1,\, \lambda_2,\,\dots,\lambda_m\}$ with $\sum_k|\lambda_k|=n.$ The skew-diagrams in $\mathbf{\lambda}$ are subdiagrams of $\mu$ which can be described as follows:
In the Young diagram of $\mu$ draw the edge joining the squares in positions $(i,j),\,(i+1,j)$ (resp.\ $(i,j),\,(i,j+1)$) if and only if $x_1v_{i,j}\neq0$ (resp.\ $x_2v_{i,j}\neq0$).
These lines divide the diagram into the skew-diagrams $\{\lambda_1,\, \lambda_2,\,\dots,\lambda_m\}$.
Such a decomposition into a disjoint union of skew-diagrams is called admissible.
Let $\lambda_1$ be the subdiagram that contains the lower left corner $(0,0)$.

\begin{lem}\label{lines}

If $\{\lambda_1,\, \lambda_2,\,\dots,\lambda_m\}$ and $m>1$ there exists $s>1$ so that $\{\lambda_1\cup\lambda_s\}\cup\{\lambda_r|2\leq r\leq m,\,r\neq s\}$ is also admissible.
\end{lem}

\begin{proof}
The commutativity of $x_1$ and $x_2$ imply that the configurations in Figure \ref{fig1} are impossible in any $2\times 2$ square.

\setlength{\unitlength}{2pt}
\begin{figure}[ht]
\begin{center}
\begin{picture}(190,20)
{\color{gray1}
\multiput(10,0)(30,0){6}{\framebox(20,20)}
\multiput(10,10)(30,0){6}{\line(1,0){20}}
\multiput(20,0)(30,0){6}{\line(0,1){20}}
}
\linethickness{1mm}
\put(10,10){\line(1,0){10}}
\put(50,10){\line(1,0){10}}
\put(110,10){\line(0,-1){10}}
\put(80,10){\line(0,1){10}}
\put(140,10){\line(0,1){10}}
\put(140,10){\line(-1,0){10}}
\put(170,10){\line(0,-1){10}}
\put(170,10){\line(1,0){10}}
\end{picture}
\end{center}
\caption{Impossible configurations}\label{fig1}
\end{figure}
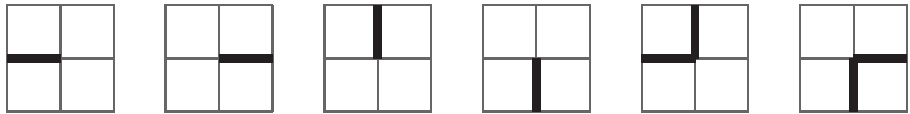

Only the configurations shown in Figure \ref{fig2} are possible in an admissible configuration.

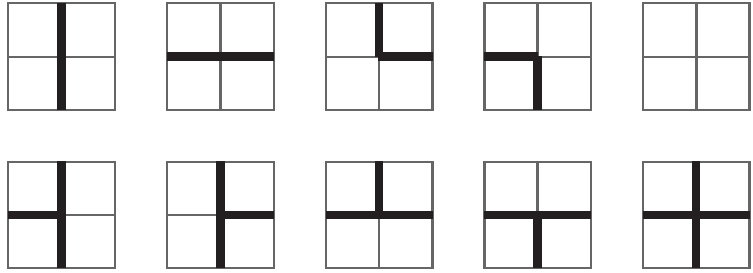
\begin{figure}[ht]
\setlength{\unitlength}{2pt}
\begin{center}
\begin{picture}(160,50)
{\color{gray1}
\multiput(10,0)(30,0){5}{\framebox(20,20)}
\multiput(10,10)(30,0){5}{\line(1,0){20}}
\multiput(20,0)(30,0){5}{\line(0,1){20}}
\multiput(10,30)(30,0){5}{\framebox(20,20)}
\multiput(10,40)(30,0){5}{\line(1,0){20}}
\multiput(20,30)(30,0){5}{\line(0,1){20}}
}
\linethickness{1mm}
\put(20,30){\line(0,1){20}}
\put(40,40){\line(1,0){20}}
\put(80,40){\line(1,0){10}}
\put(80,40){\line(0,1){10}}
\put(110,40){\line(-1,0){10}}
\put(110,40){\line(0,-1){10}}
\put(20,0){\line(0,1){20}}
\put(10,10){\line(1,0){10}}
\put(50,0){\line(0,1){20}}
\put(50,10){\line(1,0){10}}
\put(110,10){\line(1,0){10}}
\put(110,10){\line(-1,0){10}}
\put(110,10){\line(0,-1){10}}
\put(80,10){\line(1,0){10}}
\put(80,10){\line(0,1){10}}
\put(80,10){\line(-1,0){10}}
\put(140,10){\line(1,0){10}}
\put(140,10){\line(0,1){10}}
\put(140,10){\line(-1,0){10}}
\put(140,10){\line(0,-1){10}}
\end{picture}
\end{center}
\caption{Possible configurations}\label{fig2}
\end{figure}

We are to prove that it is possible to delete some lines on the boundary of $\lambda_1$ and obtain another admissible configuration.

Divide the boundary of $\lambda_1$ in the points $P_1,\,\dots P_r$ where it meets other lines. The endpoints of the boundary of $\lambda_1$ are located on the left and lower sides of the diagram of $\mu$, let $P_0$ and $P_{r+1}$ be those points. The boundary of $\lambda_1$ can be divided into the segments joining $P_{i}$ and $P_{i+1}$, for $0\leq i\leq r$.

The configurations where the points $P_0,\,\dots,\,P_{r+1}$ are shown in Figure \ref{fig3}. Note that in each of these cases there is at least one line segment on the lower left square that can be erased to obtain an admissible configuration. Those segments are dotted in Figure \ref{fig3}. Since the lines at $P_0$ and $P_{r+1}$ can be safely removed, it follows that there is a whole segment that can be removed to obtain an admissible configuration.

\begin{figure}[ht]
\setlength{\unitlength}{2pt}
\begin{center}
\begin{picture}(160,30)
{\color{gray1}
\multiput(10,0)(30,0){5}{\framebox(20,20)}
\multiput(10,10)(30,0){5}{\line(1,0){20}}
\multiput(20,0)(30,0){5}{\line(0,1){20}}
}
\multiput(20,10)(30,0){5}{\circle*{3}}
\linethickness{1mm}
\put(20,0){\line(0,1){20}}
\put(10,10){\dashbox{1}(10,0)}
\put(50,10){\line(0,1){10}}
\put(50,10){\line(1,0){10}}
\put(70,10){\dashbox{1}(10,0)}
\put(110,10){\line(1,0){10}}
\put(110,10){\line(-1,0){10}}
\put(110,0){\dashbox{1}(0,10)}
\put(80,10){\line(1,0){10}}
\put(80,10){\line(0,1){10}}
\put(50,0){\dashbox{1}(0,10)}
\put(140,10){\line(1,0){10}}
\put(140,10){\line(0,1){10}}
\put(140,0){\dashbox{1}(0,10)}
\put(130,10){\dashbox{1}(10,0)}
\end{picture}
\end{center}
\caption{Possible points of intersection}\label{fig3}
\end{figure}
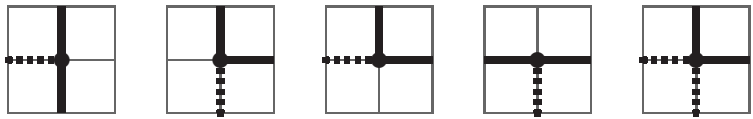

\end{proof}

\begin{examp}\label{ex1}
Let $\lambda_1$ be as shown in Figure \ref{fig4}. The segments that can be removed are the ones that join $P_2$ with $P_3$ and $P_4$ with $P_5$.

\begin{figure}[ht]
\setlength{\unitlength}{2.1pt}
\begin{center}
\begin{picture}(80,60)
{\color{gray1}
\put(0,10){\line(1,0){80}}
\put(0,20){\line(1,0){80}}
\put(0,30){\line(1,0){40}}
\put(0,40){\line(1,0){40}}
\put(0,50){\line(1,0){40}}

\put(10,0){\line(0,1){60}}
\put(20,0){\line(0,1){60}}
\put(30,0){\line(0,1){60}}
\put(40,0){\line(0,1){30}}
\put(50,0){\line(0,1){30}}
\put(60,0){\line(0,1){30}}
\put(70,0){\line(0,1){30}}
}
\linethickness{1mm}
\put(0,0){\line(1,0){80}}\put(0,0){\line(0,1){60}}
\put(0,50){\line(1,0){40}}
\put(30,30){\line(1,0){10}}
\put(30,20){\line(1,0){50}}
\put(30,20){\line(0,1){30}}
\put(60,20){\line(0,1){10}}
\put(70,0){\line(0,1){20}}

\put(1,51){$P_0$}
\put(31,51){$P_1$}
\put(31,31){$P_2$}
\put(61,21){$P_3$}
\put(71,21){$P_4$}
\put(71,1){$P_5$}
\end{picture}
\end{center}\caption{$\lambda_1$ in Example \ref{ex1}}\label{fig4}
\end{figure}
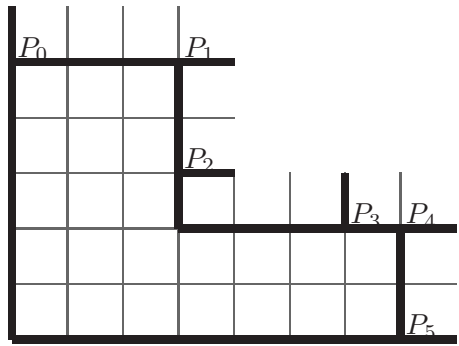
\end{examp}

\begin{proof}[Proof of Theorem \ref{thmclosure}]

Let $\mathbf{x}\in(\g_{1,0}\oplus\g_{0,1})\cap\mathfrak{C}$, we can assume that $\mathbf{x}=\bigoplus_{k\geq1}\mathbf{e}_{\lambda_k}$ with $\mathbf{\lambda}=\{\lambda_1,\, \lambda_2,\,\dots,\lambda_m\}$ as above.

We will prove that $\mathbf{x}\in\overline{(T\cdot\mathbf{e})}$ by induction on $m$. Clearly $m=1$ if an only if $\mathbf{x}\in T\cdot\mathbf{e}$. If $m>1$ we can assume that $s=2$ in Lemma \ref{lines}. We are to prove that every open set in $(\g_{1,0}\oplus\g_{0,1})\cap\mathfrak{C}^r$ containing $\mathbf{x}$ contains a pair conjugated to $\mathbf{e}_{\lambda_1\cup\lambda_2}\oplus\bigoplus_{k\geq3}\mathbf{e}_{\lambda_k}$.

Let $(x_1',x_2')=\mathbf{x'}=(\mathbf{e}_{\lambda_1\cup\lambda_2}\oplus\bigoplus_{k\geq3}\mathbf{e}_{\lambda_k})-\mathbf{x}$, i.e.,
$$x_1'v_{ij}=\begin{cases}v_{i+1,j}&\text{if }(i,j)\in\lambda_1,\,(i+1,j)\in\lambda_2\\0&\text{otherwise}
\end{cases}$$
$$x_2'v_{ij}=\begin{cases}v_{i,j+1}&\text{if }(i,j)\in\lambda_1,\,(i,j+1)\in\lambda_2\\0&\text{otherwise}
\end{cases}
$$
\end{proof}

Clearly $\mathbf{x'}\in\mathfrak{C}$ since $x_1'x_2'=x_2'x_1'=0$. Also,
$\mathbf{x'}+\mathbf{x},\mathbf{x}\in \mathfrak{C}$ and therefore
$[x_1',x_2]+[x_1,x_2']=0$. Then the line $\{\mathbf{x}+\tau\mathbf{x'}
\,|\,\tau\in\C\}$ is contained in
$(\g_{1,0}\oplus\g_{0,1})\cap\mathfrak{C}^r$ and intersects every open
set at a point other than $\mathbf{x}$. It is clear that for
$\tau\neq0$, $\mathbf{x}+\tau\mathbf{x'}$ is conjugate to
$\mathbf{e}_{\lambda_1\cup\lambda_2}\oplus\bigoplus_{k\geq3}\mathbf{e}_{\lambda_k}$.
 This
completes the induction.


\def\cprime{$'$} \def\cprime{$'$} \def\cprime{$'$} \def\cprime{$'$}
  \def\cprime{$'$} \def\cprime{$'$}

\end{document}